\documentclass[final,leqno]{siamltex}

\usepackage{amsmath}
\usepackage{stmaryrd}
\usepackage{amssymb}
\usepackage{empheq}
\usepackage{cases}
\usepackage{cite}
\usepackage{caption,lipsum}
\usepackage{epsfig}
\usepackage{graphicx}
\input psfig.sty

\newcommand{\eps}{\varepsilon}
\newcommand{\bN}{\mathbb{N}}
\newcommand{\bR}{\mathbb{R}}
\newcommand{\bC}{\mathbb{C}}

\newcommand{\mcE}{\mathcal{E}}
\newcommand{\fe}{e}

\newcommand{\bx}{\mathbf{x}}

\newcommand{\bu}{{\bf u}}
\newcommand{\bv}{{\bf v}}
\newcommand{\bz}{{\bf z}}
\newcommand{\bee}{{\bf e}}
\newcommand{\bff}{{\bf f}}
\newcommand{\bgg}{{\bf g}}
\newcommand{\be}{\begin{equation}}
\newcommand{\ee}{\end{equation}}
\newcommand{\ba}{\begin{array}}
\newcommand{\ea}{\end{array}}
\newcommand{\bea}{\begin{eqnarray}}
\newcommand{\eea}{\end{eqnarray}}
\newcommand{\beas}{\begin{eqnarray*}}
\newcommand{\eeas}{\end{eqnarray*}}

\newtheorem{remark}{Remark}[section]

\numberwithin{equation}{section}

\title{A uniformly accurate multiscale time integrator pseudospectral method
for the Klein-Gordon equation in the nonrelativistic limit regime\thanks{This
research was supported by the Singapore A*STAR SERC  PSF-Grant 1321202067.}}

\author{Weizhu Bao\thanks{Department of Mathematics,
National University of Singapore, Singapore 119076
 ({\tt matbaowz@nus.edu.sg}, {http://www.math.nus.edu.sg/\~{}bao/};
  {\tt zhxfnus@gmail.com}).} \and Yongyong Cai\thanks{ Department of Mathematics,
Purdue University, West Lafayette, IN 47907, USA (\tt cai99@math.purdue.edu).}\and
Xiaofei Zhao\footnotemark[2]}

\date{}

\begin{document}

\maketitle

\begin{abstract}
We propose and analyze a multiscale time integrator Fourier
pseudospectral (MTI-FP) method  for solving the Klein-Gordon
(KG) equation with a dimensionless parameter $0<\eps\leq1$
which is inversely proportional to the speed of light.
In the nonrelativistic limit regime, i.e. $0<\eps\ll1$,
the solution of the KG equation
propagates waves with amplitude at $O(1)$ and wavelength
at $O(\eps^2)$ in time and $O(1)$ in space,
which causes significantly numerical burdens due to
the high oscillation in time. The MTI-FP method is
designed by adapting a multiscale decomposition by
frequency (MDF) to the solution at each time step
and applying an exponential wave integrator to the nonlinear
Schr\"{o}dinger equation with wave operator under well-prepared
initial data for $\eps^2$-frequency and $O(1)$-amplitude waves
and a KG-type equation with small
initial data for the reminder waves in the MDF. We rigorously
establish two independent error bounds in $H^2$-norm to the MTI-FP method
at $O(h^{m_0}+\tau^2+\eps^2)$ and $O(h^{m_0}+\tau^2/\eps^2)$
with $h$ mesh size, $\tau$ time step
and $m_0\ge2$ an integer depending on the regularity of
the solution,  which immediately
imply that the MTI-FP converges uniformly and optimally in
space with exponential convergence rate
if the solution is smooth, and uniformly in time with linear convergence rate
at $O(\tau)$ for all  $\eps\in(0,1]$ and optimally with
quadratic convergence rate at $O(\tau^2)$
in the regimes when either $\eps=O(1)$ or $0<\eps\le \tau$.
Numerical results are reported to confirm the error bounds and
demonstrate the efficiency and accuracy
of the MTI-FP method for the KG equation, especially in
the nonrelativistic limit regime.
\end{abstract}

\begin{keywords}
Klein-Gordon equation, nonrelativistic limit, multiscale decomposition,
multiscale time integrator, uniformly accurate, meshing strategy,
exponential wave integrator, spectral method
\end{keywords}

\begin{AMS}
65M12, 65M15, 65M70, 81Q05
\end{AMS}

\pagestyle{myheadings}
\thispagestyle{plain}
\markboth{W. BAO AND X. ZHAO}{Multiscale time integrator for Klein-Gordon equation}

\section{Introduction}\label{si}
\setcounter{equation}{0}
In this paper, we consider the dimensionless Klein-Gordon (KG)
equation in $d$ dimensions $(d=1,2,3)$
\cite{Dong,Machihara1,Machihara2,Masmoudi,Sakurai,Faou,Ginibre1,Ginibre2,Strauss,Tao,Reich}:
\begin{equation}\label{KG}
\left\{
  \begin{split}
    & \eps^2\partial_{tt}u-\Delta u+\frac{1}{\eps^2}u
    +f\left(u\right)=0,\quad \mathbf{x}\in\bR^d,\quad t>0,\\
    & u(\mathbf{x},0)= \phi_1(\mathbf{x}),\quad\partial_tu(\mathbf{x},0)
    =\frac{1}{\eps^2}\phi_2(\mathbf{x}),\quad \bx\in\bR^d.
  \end{split}
\right.
\end{equation}
Here $t$ is time, $\mathbf{x}$ is the spatial coordinate,
$u:=u(\mathbf{x},t)$ is a complex-valued
scalar field, $0<\eps\leq1$ is a dimensionless parameter which is
inversely proportional to the speed of light,
 $\phi_1$ and $\phi_2$ are
two given complex-valued initial data which are independent of $\eps$,
and $f(u): \bC\rightarrow\bC$ is a given {\sl gauge invariant} nonlinearity
which is independent of $\eps$  and satisfies \cite{Zhao,Machihara1,Machihara2,Masmoudi,Sakurai,Faou}
\begin{equation}\label{gauge}
f(\fe^{is}u)=\fe^{is}f(u), \qquad \forall s\in \bR.
\end{equation}
We remak that when the initial data $\phi_1(\bx),\phi_2(\bx): \ \bR^d \to\bR$ and $\mathbf{f}(u):\ \bR \to\bR$,
then the solution $u(\bx,t)$ of (\ref{KG}) is real-valued. In this case,
the gauge invariant condition (\ref{gauge}) for the nonlinearity in
(\ref{KG}) is no longer needed. Thus (\ref{KG}) includes the classical KG equation
with the solution $u$ real-valued as a special case \cite{cao,Duncan,Morawetz,Segal,Simon,Strauss,Tao}.

The above KG equation is also known as the relativistic version of the
Schr\"{o}dinger equation used to describe the
dynamics of a spinless particle \cite{Sakurai}. In most applications \cite{Ginibre1,Ginibre2,Glassey1,Glassey2,Machihara2, Masmoudi,Pecher,Faou,Reich},
$f(u)$ is taken as the {\sl pure power} nonlinearity, i.e.
\begin{equation}\label{power}
f(u)=g(|u|^2)u, \ \hbox{with}\  g(\rho)=\lambda \rho^p\ \hbox{for some}
\ \lambda\in \bR, \ p\in\bN_0:=\bN\cup\{0\},
\end{equation}
and then the KG equation (\ref{KG}) conserves the energy \cite{Dong,Ginibre1,Ginibre2,Masmoudi}
\begin{eqnarray}\label{energy}
E(t)&:=&\int_{\bR^d}\left[\eps^2|\partial_t u(\bx,t)|^2+|\nabla u(\bx,t)|^2+\frac{1}{\eps^2}|u(\bx,t)|^2+F(|u(\bx,t)|^2)\right]d \bx\nonumber\\
&\equiv&\int_{\bR^d}\left[\frac{1}{\eps^2}|\phi_2(\bx)|^2+|\nabla \phi_1(\bx)|^2+\frac{1}{\eps^2}|\phi_1(\bx)|^2+F(|\phi_1(\bx)|^2)\right]d \bx=E(0),\ t\ge0,
\end{eqnarray}
where $F(\rho)=\int_0^\rho g(s)\,ds$.

For a fixed $\eps=\eps_0=O(1)$, i.e. O(1)-speed of light regime (e.g. $\eps=1$), the KG equation (\ref{KG})
has been studied extensively in both analytical and numerical aspects,
see \cite{cao,Cohen2,Duncan,Ginibre1,Ginibre2,Glassey1,Glassey2,Jimenez,
Morawetz,Pecher,Segal,Simon,Strauss,Tao}  and references therein.
Recently, more attentions have been devoted to analyzing the solution structure \cite{Machihara1,Machihara2,Masmoudi,Najman,Tsutsumi} and
designing efficient and accurate numerical methods \cite{Dong,Faou} of the problem (\ref{KG})
in the nonrelativistic limit regime, i.e. $0<\eps\ll1$. In fact, due to
that the energy $E(t)=O(\eps^{-2})$ in (\ref{energy}) becomes unbounded when
$\eps\to0$, this brings significant difficulties in the mathematical analysis of
the problem (\ref{KG}) in the nonrelativistic limit regime. Based on
recent analytical results \cite{Machihara1,Machihara2,Masmoudi,Najman,Tsutsumi}, the problem (\ref{KG})
propagates waves with amplitude at $O(1)$ and wavelength
at $O(\eps^2)$ and $O(1)$ in time and space, respectively, when $0<\eps\ll1$.
To illustrate this, Fig. \ref{fig:0} shows the solution of the KG equation (\ref{KG})
with $d=1$, $f(u)=|u|^2u$, $\phi_1(x)=e^{-x^2/2}$ and $\phi_2(x)=\frac{3}{2}\phi_1(x)$ for
different $\eps$.
\begin{figure}
\centerline{\psfig{figure=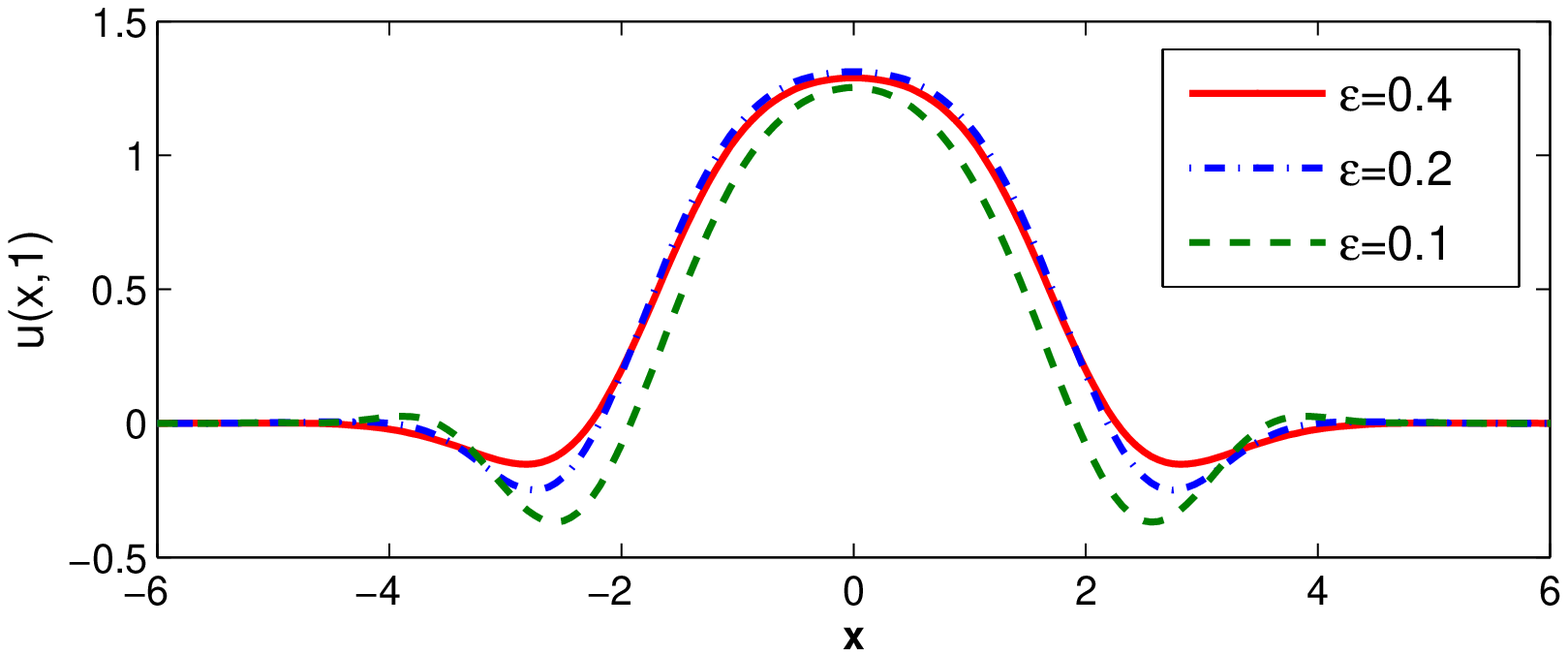,height=6cm,width=14cm}}
\centerline{\psfig{figure=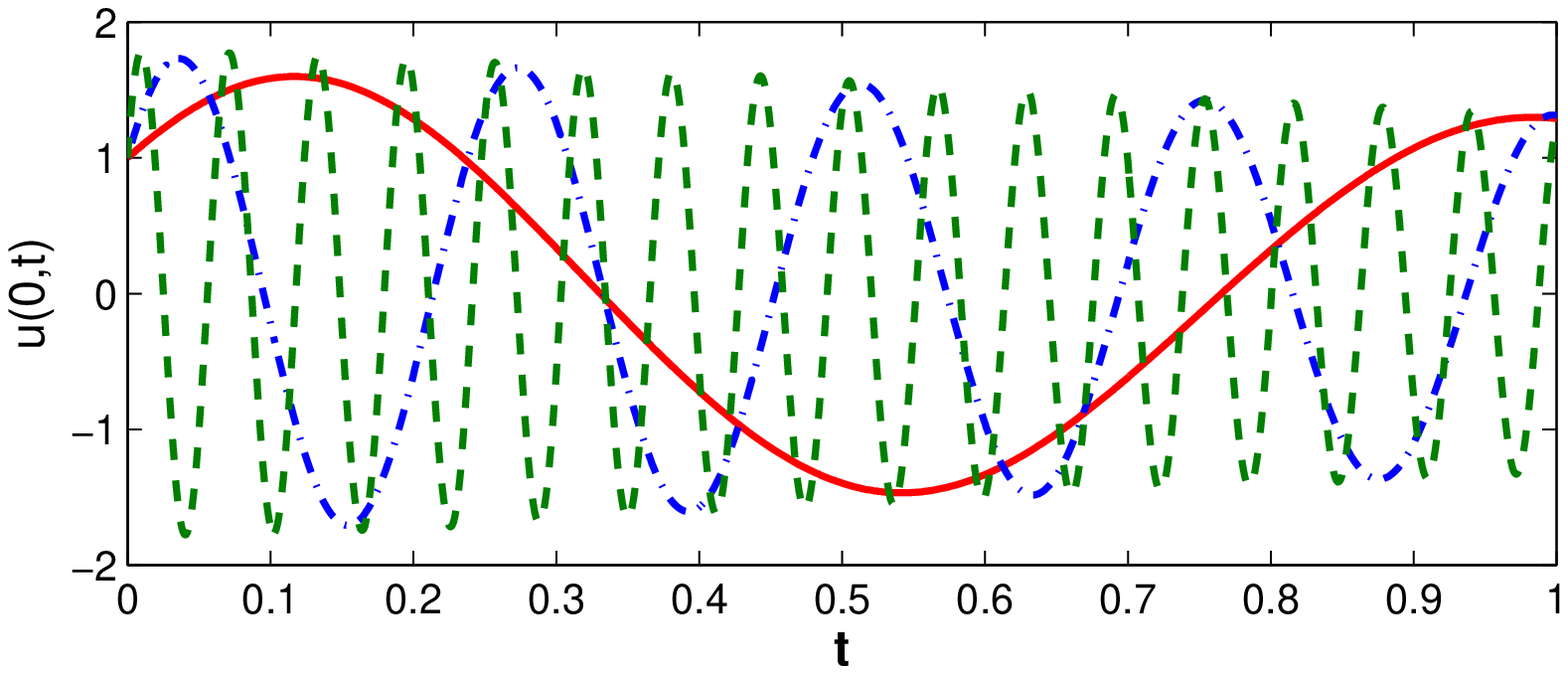,height=6cm,width=14cm}}
\caption{The solution of (\ref{KG}) with $d=1$, $f(u)=|u|^2u$, $\phi_1(x)=e^{-x^2/2}$ and
$\phi_2(x)=\frac{3}{2}\phi_1(x)$ for different $\eps$.}\label{fig:0}
\end{figure}
This highly oscillatory nature of the solution of (\ref{KG})
causes severe burdens in practical computation, making the numerical
approximation of (\ref{KG}) extremely challenging and costly in the regime of
$0<\eps\ll 1$. Different numerical methods, including finite difference
time domain (FDTD) methods \cite{Dong}, exponential wave integrator Fourier pseudospectral (EWI-FP)
method \cite{Dong} and asymptotic preserving (AP) method \cite{Faou} have been proposed and analyzed
as well as compared for solving (\ref{KG}) numerically  in the nonrelativistic limit regime.
In the above numerical study, they paid particular attention on
the resolution of different numerical methods, i.e. meshing strategy requirement
(or $\eps$-scalability) for (\ref{KG}) when $0<\eps\ll1$. Based on their results,
in order to capture `correctly' the oscillatory solution of (\ref{KG})
in practical computations, the frequently used FDTD methods
request mesh size $h=O(1)$ and time step $\tau=O(\eps^3)$
and the EWI-FP methods require $h=O(1)$ and $\tau=O(\eps^2)$, when
$0<\eps\ll1$ \cite{Dong}. Thus the FDTD and EWI-FP methods converges optimally in space and time
for any fixed $\eps=\eps_0=O(1)$, but they do not converge when $\tau=O(\eps)$.  Similarly,
the AP method \cite{Faou} requests $h=O(1)$ and $\tau=O(1)$ when
$0<\eps\ll1$, but it brings $O(1)$-error when $\eps=\eps_0=O(1)$.
Thus all the above numerical methods for the problem (\ref{KG})
do not converge uniformity for $\eps\in(0,1]$ \cite{Dong,Faou}.
Recently, uniformly accurate numerical schemes for high oscillatory Klein-Gordon and nonlinear
Schr\"{o}dinger equations have been proposed and analyzed \cite{Chartier} based on
embedding the problem in a suitable ``two-scale'' reformulation with the
induction of an additional variable and using the Chapman-Enskog expansion to separate
the fast time scale and the slow one. 

Very recently, by using a highly oscillatory second-order ordinary differential equation (ODE)
which has the same oscillatory nature as the problem (\ref{KG}) in time, we proposed and analyzed
two multiscale time integrators (MTIs) based on multiscale decompositions of the solution \cite{Zhao}.
The two MTIs converge uniformly for $\eps\in(0,1]$ and have some advantages compared to
the FDTD and EWI as well as asymptotic preserving methods in integrating
highly oscillatory second-order ODEs for $\eps\in(0,1]$ \cite{Zhao}, especially
when $\eps$ is not too big and too small, i.e. in the intermediate regime.
The aim of this paper is to design and analyze a multiscale time integrator Fourier pseudospectral (MTI-FP)
method for the problem (\ref{KG}) based on a multiscale decomposition
of the solution of (\ref{KG}) \cite{Machihara1,Machihara2,Masmoudi,Najman}
and the MTIs to the highly oscillatory second-order ODEs in \cite{Zhao}. The proposed MTI-FP method to (\ref{KG})
is explicit, efficient and accurate in practical computation,
and converges in time uniformly  at linear convergence rate
for all $\eps\in(0,1]$ and optimally at quadratic convergence rate
in the regimes $\eps=O(1)$ or $0<\eps\leq\tau$. Thus our method
is different with those numerical methods in \cite{Dong,Chartier,Faou}.

The paper is organized as follows. In section \ref{sec: MD},
we introduce a multiscale decomposition for the KG equation (\ref{KG}) based on frequency.
A MTI-FP method is proposed in section \ref{sec: MTI}, and its rigorous error bounds are established
in section \ref{sec: convergence}. Numerical results are reported in section \ref{sec: num results}. Finally, some conclusions are drawn in section \ref{sec: conlusion}. Throughout this paper,
we adopt the standard Sobolev spaces \cite{Adams} and use
the notation $A\lesssim B$ to represent that there exists a generic constant $C>0$,
which is independent of time step $\tau$ (or $n$), mesh size $h$ and $\eps$, such that $|A|\leq CB$.

\section{A multiscale decomposition}\label{sec: MD}

Let $\tau=\Delta t>0$ be the step size, and denote time steps by $t_n=n\tau$ for $n=0,1,\ldots$\;.
In this section, we present a multiscale decomposition
for the solution of (\ref{KG})
on the time interval $[t_n, t_{n+1}]$
with given initial data at $t=t_n$ as
\begin{equation}\label{Initial}
u(\mathbf{x},t_n)=\phi_1^n(\mathbf{x})=O(1),\qquad \partial_tu(\mathbf{x},t_n)=\frac{1}{\eps^2}\phi_2^n(\mathbf{x})=O\left(\frac{1}{\eps^{2}}\right).
\end{equation}
Similarly to the analytical study of the nonrelativistic limit of the nonlinear KG equation
in \cite{Machihara2,Masmoudi}, we take an ansatz to the solution
$u(\mathbf{x},t):=u(\mathbf{x},t_n+s)$ of (\ref{KG}) on the
time interval $[t_n, t_{n+1}]$  with (\ref{Initial}) as \cite{Zhao}
\begin{equation}\label{ansatz}
u(\mathbf{x},t_n+s)=\fe^{is/\eps^2}z_+^n(\mathbf{x},s)+\fe^{-is/\eps^2}\overline{z_-^n}(\mathbf{x},s)
+r^n(\mathbf{x},s), \quad \bx\in\bR^d,\ \ 0\leq s\leq\tau.
\end{equation}
Here and after, $\bar{z}$ denotes the complex conjugate of a complex-valued function $z$.
Differentiating (\ref{ansatz}) with respect to $s$, we have
\bea\label{ansatad}
\partial_s u(\mathbf{x},t_n+s)&=&\fe^{is/\eps^2}\left[\partial_s z_+^n(\mathbf{x},s)+\frac{i}{\eps^2}z_+^n(\mathbf{x},s)\right]+\fe^{-is/\eps^2}
\left[\partial_s\overline{ z_-^n}(\mathbf{x},s)-\frac{i}{\eps^2}\overline{z_-^n}(\mathbf{x},s)\right]\nonumber\\
&&+\partial_sr^n(\mathbf{x},s),\qquad \bx\in\bR^d,\ \ 0\leq s\leq\tau.
\eea
Plugging (\ref{ansatz}) into (\ref{KG}), we get for $\bx\in\bR^d$ and $0\leq s\leq\tau$
\beas
&&\fe^{is/\eps^2}\left[\eps^2\partial_{ss}{z}_+^n(\mathbf{x},s)+2i\partial_sz_+^n(\mathbf{x},s)
-\Delta z_+^n(\mathbf{x},s) \right]+\eps^2\partial_{ss}r^n(\mathbf{x},s)+\Delta r^n(\mathbf{x},s)\nonumber\\
&&+\fe^{-is/\eps^2}\left[\eps^2\partial_{ss}\overline{z_-^n}(\mathbf{x},s)
-2i\partial_s\overline{z_-^n}(\mathbf{x},s)
-\Delta \overline{z_-^n}(\mathbf{x},s)\right]+\frac{r^n(\mathbf{x},s)}{\eps^2}
+f(u(\mathbf{x},t_n+s))=0.
\eeas
Multiplying the above equation by $\fe^{-is/\eps^2}$ and $\fe^{is/\eps^2}$, respectively, we can decompose
it into a coupled system for two $\eps^2$-frequency waves with the unknowns
$z_\pm^n(\bx,s):=z_\pm^n$
and the rest frequency waves with the unknown $r^n(\bx,s):=r^n$ as
\begin{equation}\label{MDF}
\left\{
  \begin{split}
&\eps^2\partial_{ss}z_\pm^n+2i\partial_s z_\pm^n-\Delta z_\pm^n
+f_\pm\left(z_+^n,z_-^n\right)=0, \\
&\eps^2\partial_{ss}r^n-\Delta r^n+\frac{1}{\eps^2}r^n+f_r\left(z_+^n,z_-^n,r^n;s\right)=0,
\end{split}
\right. \qquad \bx\in\bR^d,\quad 0\leq s\leq\tau,
\ee
where
{\small \beas
&&f_\pm\left(z_+,z_{-}\right)=\frac{1}{2\pi}\int_0^{2\pi} f\left(z_\pm+\fe^{i\theta}\overline{z_\mp}\right)
d\theta, \qquad z_\pm,r\in {\mathbb C}, \qquad 0\le s\le \tau,\\
&&f_r\left(z_+,z_-,r;s\right)=f\left(\fe^{is/\eps^2}z_++\fe^{-is/\eps^2}
\overline{z_-}+r\right)-f_+\left(z_+,z_-\right)\fe^{is/\eps^2}-\overline{f_-}\left(z_+,z_-\right)\fe^{-is/\eps^2}.
\eeas }
In order to find proper initial conditions for the above system (\ref{MDF}),
setting $s=0$ in (\ref{ansatz}) and (\ref{ansatad}),
noticing (\ref{Initial}), we obtain
\begin{equation}\label{init123}
\left\{
\begin{split}
 &z_+^n(\bx,0)+\overline{z_-^n}(\bx,0)+r^n(\bx,0)=\phi_1^n(\bx),\quad \bx\in\bR^d,\\
 &\frac{i}{\eps^2}\left[z_+^n(\bx,0)-\overline{z_-^n}(\bx,0)\right]+\partial_s z_+^n(\bx,0)+\partial_s\overline{z_-^n}(\bx,0)
 +\partial_sr^n(\bx,0)=\frac{\phi_2^n(\bx)}{\eps^2}.
\end{split}
\right.
\end{equation}
Now we decompose the above initial data so as to: (i) equate $O\left(\frac{1}{\eps^2}\right)$ and $O(1)$ terms
in the second equation of (\ref{init123}), respectively, and (ii) be well-prepared for the first two equations  in (\ref{MDF})
when $0<\eps\ll 1$, i.e. $\partial_sz_+^n(\bx,0)$ and $\partial_sz_-^n(\bx,0)$ are determined
from the first two equations in (\ref{MDF}), respectively, by setting $\eps=0$ and $s=0$ \cite{Cai1,Cai2}:
\begin{equation}\label{FSW-i1}
\left\{
  \begin{split}
&z_+^n(\bx,0)+\overline{z_-^n}(\bx,0)=\phi_1^n(\bx),\qquad i\left[z_+^n(\bx,0)-\overline{z_-^n}(\bx,0)\right]=\phi_2^n(\bx),\\
&2i\partial_sz_\pm^n(\bx,0)-\Delta z_\pm^n(\bx,0)+f_\pm\left(z_+^n(\bx,0),z_-^n(\bx,0)\right)=0,
\qquad \qquad \quad \bx\in\bR^d,\\
&r^n(\bx,0)=0, \qquad \partial_sr^n(\bx,0)+\partial_sz_+^n(\bx,0)+\partial_s\overline{z_-^n}(\bx,0)=0.
\end{split}
  \right.
\end{equation}
Solving (\ref{FSW-i1}), we get the initial data for (\ref{MDF}) as
\begin{equation}\label{FSW-i21}
\left\{
  \begin{split}
&z_+^n(\bx,0)=\frac{1}{2}\left[\phi_1^n(\bx)-i\phi_2^n(\bx)\right], \qquad  z_-^n(\bx,0)=\frac{1}{2}\left[\overline{\phi_1^n}(\bx)-i\ \overline{\phi_2^n}(\bx)\right],\\
&\partial_sz_\pm^n(\bx,0)=\frac{i}{2}\left[-\Delta z_\pm^n(\bx,0)+
f_\pm\left(z_+^n(\bx,0),z_-^n(\bx,0)\right)\right],\qquad \qquad\bx\in\bR^d,\\
&r^n(\bx,0)=0, \qquad \partial_sr^n(\bx,0)=-\partial_sz_+^n(\bx,0)-\partial_s\overline{z_-^n}(\bx,0).
\end{split}
  \right.
\end{equation}
The above decomposition (\ref{ansatz}) can be called as multiscale decomposition by frequency (MDF).
In fact, it can also be regarded as to decompose slow waves at $\eps^2$-wavelength and fast waves at other wavelengths,
thus it can also be called as fast-slow frequency decomposition (FSFD). On the other hand,
the amplitude of $z_\pm^n$ is usually at $O(1)$ and the amplitude of $r^n$ is at $O(\eps^2)=o(1)$
when $\eps$ is small, thus it can also be regarded as large-small amplitude decomposition (LSAD).
Specifically, for the pure power nonlinearity, i.e. $f$ satisfies (\ref{power}),
explicit formulas for $f_\pm$ and $f_r$ have been given in \cite{Zhao}.

After solving the decomposed system (\ref{MDF})  with the initial data (\ref{FSW-i21}), we get $z_\pm^n(\bx,\tau)$, $\partial_s z_\pm^n(\bx,\tau)$, $r^n(\bx,\tau)$ and $\partial_s r^n(\bx,\tau)$.
Then we can reconstruct the solution to (\ref{KG}) at $t=t_{n+1}$ by setting $s=\tau$ in (\ref{ansatz}) and (\ref{ansatad}), i.e.,
\begin{equation}\label{u n+1}
\left\{
\begin{split}
&u(\bx,t_{n+1})=\fe^{i\tau/\eps^2}z_+^n(\bx,\tau)+\fe^{-i\tau/\eps^2}\overline{z_-^n}(\bx,\tau)
+r^n(\bx,\tau):=\phi_1^{n+1}(\bx),\\
&\partial_t u(\bx,t_{n+1})=\frac{1}{\eps^2}\phi_2^{n+1}(\bx),\qquad\bx\in\bR^d,\\
\end{split}\right.
\end{equation}
with
{\small
\begin{equation*}
\phi_2^{n+1}(\bx)=\fe^{i\tau/\eps^2}\left[\eps^2\partial_sz_+^n(\bx,\tau)+i z_+^n(\bx,\tau)\right]+\fe^{-i\tau/\eps^2}
\left[\eps^2\partial_s\overline{z_-^n}(\bx,\tau)-i\overline{z_-^n}(\bx,\tau)\right]
+\eps^2\partial_sr^n(\bx,\tau).
\end{equation*}}

\section{A MTI-FP method}\label{sec: MTI}

In this section, based on the MDF (\ref{MDF}), we propose a new numerical method for
solving the KG equation (\ref{KG}) with the pure power nonlinearity (\ref{power}),
which is uniformly accurate for $\eps\in(0,1]$.
For the simplicity of notations, we present the numerical method
in one space dimension (1D) with a cubic nonlinearity, i.e. $d=1$ in (\ref{KG}) and
$f(u)=\lambda|u|^{2}u$ with $\lambda\in\bR$ a given constant in (\ref{power}).
In this case, we have
\begin{equation}\label{fr cubic}
\left\{
\begin{split}
&f_\pm\left(z_+,z_{-}\right)=\lambda\left(|z_\pm|^2+2|z_\mp|^2\right)z_\pm,
\qquad z_\pm,r\in {\mathbb C}, \qquad 0\le s\le \tau,\\
&f_r\left(z_+,z_-,r;s\right)=\fe^{3is/\eps^2}g_+(z_+,z_-)+\fe^{-3is/\eps^2}
\overline{g_-}(z_+,z_-)+w(z_+,z_-,r;s), \\
\end{split}\right.
\end{equation}
with
\begin{equation}\label{g_pm cubic}
\left\{
\begin{split}
 &g_\pm(z_+,z_-)=\lambda\,{z^2_\pm}\,z_\mp,\qquad z_\pm,r\in {\mathbb C}, \qquad 0\le s\le \tau,\\
 &w(z_+,z_-,r;s)=f\left(\fe^{is/\eps^2}z_++\fe^{-is/\eps^2}\overline{z_-}+r\right)
 -f\left(\fe^{is/\eps^2}z_++\fe^{-is/\eps^2}\overline{z_-}\right).\\
 \end{split}\right.
\end{equation}
Generalizations to higher dimensions and general pure power  nonlinearity  are straightforward and all the results presented in this paper are still valid with minor modifications. Due to fast decay of the
solution to the KG equation (\ref{KG}) at far field, similar to those in the literature for
numerical computations \cite{Dong,cao,Cohen2,Duncan,Faou,Jimenez,Strauss},
the whole space problem (\ref{KG}) in 1D is usually truncated
onto a finite interval $\Omega=(a,b)$  with periodic boundary conditions
($a$ and $b$ are usually chosen sufficient large such that
the truncation error is negligible):
\begin{equation}\label{KG trun}
\left\{
\begin{split}
& \eps^2\partial_{tt} u(x,t)-\partial_{xx} u(x,t)+\frac{u(x,t)}{\eps^2}
+f\left(u(x,t)\right)=0,\quad x\in\Omega=(a,b),\ t>0,\\
& u(a,t)=u(b,t),\quad \partial_x u(a,t)=\partial_x u(b,t), \qquad t\geq0,\\
& u(x,0)= \phi_1(x),\quad\partial_tu(x,0)=\frac{1}{\eps^2}\phi_2(x),\qquad x\in\overline{\Omega}=[a,b].
  \end{split}
\right.
\end{equation}
Consequently, for $n\ge0$, the decomposed system MDF (\ref{MDF}) in 1D collapses to
\begin{equation}\label{MDF trun}
\left\{
\begin{split}
&\eps^2\partial_{ss}z_\pm^n+2i\partial_s z_\pm^n-\partial_{xx}z_\pm^n+f_\pm\left(z_+^n,z_-^n\right)=0, \\
&\eps^2\partial_{ss}r^n-\partial_{xx} r^n+\frac{1}{\eps^2}r^n+f_r\left(z_+^n,z_-^n,r^n;s\right)=0,  \quad
a<x<b, \ 0<s\leq\tau.
\end{split}
\right.
\end{equation}
The initial and boundary conditions for the above system are
\begin{equation}\label{MDF ini}
\left\{
\begin{split}
&z_\pm^n(a,s)=z_\pm^n(b,s),\qquad \partial_x z_\pm^n(a,s)=\partial_x z_\pm^n(b,s), \\
&r^n(a,s)=r^n(b,s),\qquad \ \partial_x r^n(a,s)=\partial_x r^n(b,s), \qquad 0\leq s\leq\tau;\\
&z_+^n(x,0)=\frac{1}{2}\left[\phi_1^n(x)-i\phi_2^n(x)\right], \quad z_-^n(x,0)=\frac{1}{2}\left[\overline{\phi_1^n}(x)-i\,\overline{\phi_2^n}(x)\right],\\
&\partial_s z_\pm^n(x,0)=\frac{i}{2}\left[-\partial_{xx} z_\pm^n(x,0)
+f_\pm\left(z_+^n(x,0),z_-^n(x,0)\right)\right], \qquad a\le x\le b,\\
&r^n(x,0)=0, \qquad \partial_s r^n(x,0)=-\partial_s z_+^n(x,0)-\partial_s\,\overline{z_-^n}(x,0).
\end{split}
\right.
\end{equation}

In order to discretize (\ref{MDF trun}) with (\ref{MDF ini}), we first apply the Fourier
spectral method in space and then use the exponential wave integrator (EWI)  for
time integration \cite{Zhao}. Choose the mesh size $h:=\Delta x=(b-a)/N$ with $N$ a positive integer
 and denote grid points as
$x_j:=a+jh$ for  $j=0,1,\ldots, N$. Define
\beas
&&X_N:=\mbox{span}\left\{\phi_l(x)=\fe^{i\mu_l(x-a)}\ |\ l=-\frac{N}{2}, \ldots, \frac{N}{2}-1\right\}
\quad \hbox{with}\ \mu_l=\frac{2\pi l}{b-a}, \\
&&Y_N:=\left\{\bv=(v_0,v_1,\ldots,v_N)\in\bC^{N+1} \; |\; v_0=v_N\right\}\quad \hbox{with} \
 \|\bv\|_{l^2}=h\sum_{j=0}^{N-1}|v_j|^2.
\eeas
For a periodic function $v(x)$ on $\overline{\Omega}$ and a vector $\bv\in Y_N$,
let $P_N: L^2(\Omega)\rightarrow X_N$ be the standard $L^2$-projection operator,
 and $I_N: C(\Omega)\rightarrow X_N$ or  $Y_N \rightarrow X_N$
be the trigonometric interpolation operator \cite{Shen}, i.e.
\begin{equation}\label{project oper}
(P_Nv)(x)=\sum_{l=-N/2}^{N/2-1}\widehat{v}_l\;\fe^{i\mu_l(x-a)},\quad (I_N\bv)(x)=\sum_{l=-N/2}^{N/2-1}\widetilde{v}_l\;\fe^{i\mu_l(x-a)},\quad a\leq x\leq b,
\end{equation}
where $\widehat{v}_l$ and $\widetilde{v}_l$ are the Fourier and discrete Fourier transform
coefficients of the periodic function $v(x)$ and vector $\bv$, respectively, defined as
\begin{equation}\label{Fourier tans}
\widehat{v}_l=\frac{1}{b-a}\int_a^b v(x)\;\fe^{-i\mu_l(x-a)}dx,\qquad \widetilde{v}_l=\frac{1}{N}\sum_{j=0}^{N-1}v_j\;\fe^{-i\mu_l(x_j-a)}.
\end{equation}

Then a Fourier spectral method for discretizing (\ref{MDF trun}) reads:\\
Find $z_{\pm,N}^n:=z_{\pm,N}^n(x,s),\ r^n_N:=r^n_N(x,s)\in X_N$ for $0\le s\leq\tau$, i.e.
\begin{equation} \label{z r PN}
z_{\pm,N}^n(x,s)=\sum_{l=-N/2}^{N/2-1}\widehat{(z_\pm^n)}_l(s)\fe^{i\mu_l(x-a)},\quad r_N^n(x,s)=\sum_{l=-N/2}^{N/2-1}\widehat{(r^n)}_l(s)\fe^{i\mu_l(x-a)},
\end{equation}
such that for $0< s< \tau$
\begin{equation}\label{MDF PN}
\left\{
\begin{split}
&\eps^2\partial_{ss}z_{\pm,N}^n+2i\partial_s z_{\pm,N}^n-\partial_{xx}z_{\pm,N}^n+P_Nf_\pm\left(z_{+,N}^n,z_{-,N}^n\right)=0,\quad a< x< b,\\
&\eps^2\partial_{ss}r_N^n-\partial_{xx}r_N^n+\frac{1}{\eps^2}r_N^n+P_Nf_r\left(z_{+,N}^n,
z_{-,N}^n,r_N^n;s\right)=0.
\end{split}
\right.
\end{equation}
Substituting (\ref{z r PN}) into (\ref{MDF PN}) and noticing the orthogonality of $\phi_l(x)$, we get
\begin{equation}\label{MDF l}
\left\{
\begin{split}
&\eps^2\widehat{(z_{\pm}^n)}_l''(s)+2i\widehat{(z_\pm^n)}_l'(s)+\mu_l^2\widehat{(z_\pm^n)}_l(s)
+\widehat{(f_\pm^n)}_l(s)=0,\quad 0<s\leq\tau,\\
&\eps^2\widehat{(r^n)}_l''(s)+\left(\mu_l^2+\frac{1}
{\eps^2}\right)\widehat{(r^n)}_l(s)+\widehat{(f_r^n)}_l(s)=0,\quad -\frac{N}{2}\le l\le \frac{N}{2}-1,
\end{split}
\right.
\end{equation}
where $\widehat{(f_\pm^n)}_l(s)$ and $\widehat{(f_r^n)}_l(s)$ are the Fourier coefficients of $f_\pm^n(x,s):=f_\pm(z_{+,N}^n(x,s)$,  $z_{-,N}^n(x,s))$ and
$f_r^n(x,s):=f_r\left(z_{+,N}^n(x,s),z_{-,N}^n(x,s),r_N^n(x,s);s\right)$, respectively.
In order to apply the EWIs for integrating (\ref{MDF l}) in time, for each fixed $-N/2\le l\le N/2-1$,
we re-write (\ref{MDF l}) by using  the variation-of-constant formulas
\begin{equation}\label{VCF}
\left\{
\begin{split}
&\widehat{(z_\pm^n)}_l(s)=a_l(s)\widehat{(z_\pm^n)}_l(0)+\eps^2b_l(s)\widehat{(z_\pm^n)}_l'(0)
-\int_0^{s}b_l(s-\theta)\widehat{(f_\pm^n)}_l(\theta)\, d\theta,\\
&\widehat{(r^n)}_l(s)=\frac{\sin(\omega_l s)}{\omega_l}\widehat{(r^n)}_l'(0)
-\int_0^s\frac{\sin\left(\omega_l(s-\theta)\right)}{\eps^2\omega_l}\widehat{(f_r^n)}_l(\theta)\,d\theta,
\quad 0\leq s\leq \tau,
\end{split}
\right.
\end{equation}
where $\omega_l=\frac{1}{\eps^2}\sqrt{1+\mu_l^2\eps^2}$ and 
\begin{equation}\label{nd pm def}
\left\{
\begin{split}
&a_l(s):=\frac{\lambda^+_l\fe^{is\lambda^-_l}-\lambda^-_l\fe^{is\lambda^+_l}}{\lambda^+_l-\lambda^-_l},\quad\ \ b_l(s):=i\frac{\fe^{is\lambda^+_l}-
\fe^{is\lambda^-_l}}{\eps^2(\lambda^-_l-\lambda^+_l)}, \quad 0\le s\le \tau,\\
&\lambda^+_l=-\frac{1+\sqrt{1+\mu_l^2\eps^2}}{\eps^2}=O\left(\frac{1}{\eps^2}\right),\quad
\lambda^-_l=-\frac{1-\sqrt{1+\mu_l^2\eps^2}}{\eps^2}=O\left(1\right).
\end{split}
\right.
\end{equation}
Differentiating (\ref{VCF}) with respect to $s$, we obtain
\begin{equation}\label{dz dr}
\left\{
\begin{split}
&\widehat{(z_\pm^n)}_l'(s)=a_l'(s)\widehat{(z_\pm^n)}_l(0)+\eps^2b_l'(s)\widehat{(z_\pm^n)}_l'(0)-\int_0^s b_l'(s-\theta)\widehat{(f_\pm^n)}_l(\theta)\, d\theta,\\
&\widehat{(r^n)}_l'(s)=\cos(\omega_l s)\widehat{(r^n)}_l'(0)-\int_0^s\frac{\cos\left(\omega_l(s-\theta)\right)}{\eps^2}
\widehat{(f_r^n)}_l(\theta)\,d\theta,\quad 0\leq s\leq \tau,
\end{split}
\right.
\end{equation}
where
\begin{equation}
a_l'(s)=i\lambda^+_l\lambda^-_l\frac{\fe^{is\lambda^-_l}-\fe^{is\lambda^+_l}}{\lambda^+_l-\lambda^-_l},\quad\ b_l'(s)=\frac{\lambda^+_l\fe^{is\lambda^+_l}-
\lambda^-_l\fe^{is\lambda^-_l}}{\eps^2(\lambda^+_l-\lambda^-_l)}, \quad\ 0\le s\le \tau.
\end{equation}
Taking $s=\tau$ in (\ref{VCF}) and  (\ref{dz dr}), noticing (\ref{fr cubic}) and (\ref{g_pm cubic}),
and approximating the integrals
either by the Gautschi's type quadrature \cite{Zhao,Gaustchi,Lubich} or by the
standard trapezoidal rule \cite{Zhao,Gaustchi,Lubich}, we get
\begin{equation}\label{z r app}
\left\{
\begin{split}
&\widehat{(z_\pm^n)}_l(\tau)\approx a_l(\tau)\widehat{(z_\pm^n)}_l(0)+\eps^2b_l(\tau)\widehat{(z_\pm^n)}_l'(0)
-c_l(\tau)\widehat{(f_\pm^n)}_l(0)-d_l(\tau)\widehat{(f_\pm^n)}_l'(0),\\
&\widehat{(r^n)}_l(\tau)\approx\frac{\sin(\omega_l \tau)}{\omega_l}\widehat{(r^n)}_l'(0)
-p_l(\tau)\widehat{(g_+^n)}_l(0)-q_l(\tau)\widehat{(g_+^n)}_l'(0)\\
&\qquad \qquad \quad -\overline{p_l}(\tau)\widehat{\left(\overline{g_-^n}\right)}_l(0)
-\overline{q_l}(\tau)\widehat{\left(\overline{g_-^n}\right)}_l'(0),\\
&\widehat{(z_\pm^n)}_l'(\tau)\approx a_l'(\tau)\widehat{(z_\pm^n)}_l(0)+\eps^2b_l'(\tau)\widehat{(z_\pm^n)}_l'(0)
-c_l'(\tau)\widehat{(f_\pm^n)}_l(0)-d_l'(\tau)\widehat{(f_\pm^n)}_l'(0),\\
&\widehat{(r^n)}_l'(\tau)\approx\cos(\omega_l \tau)\widehat{(r^n)}_l'(0)-p_l'(\tau)\widehat{(g_+^n)}_l(0)-q_l'(\tau)\widehat{(g_+^n)}_l'(0)
-\overline{p_l'}(\tau)\widehat{\left(\overline{g_-^n}\right)}_l(0)\\
&\qquad\quad\quad\ \ -\overline{q_l'}(\tau)\widehat{\left(\overline{g_-^n}\right)}_l'(0)-\frac{\tau}{2\eps^2}
\widehat{(w^n)}_l(\tau),
\end{split}
\right.
\end{equation}
where $\widehat{(g_\pm^n)}_l(s)$, $\widehat{(w^n)}_l(s)$,
$\widehat{(f_\pm^n)}_l'(s)=\widehat{(\partial_sf_\pm^n)}_l(s)$ and $\widehat{(g_\pm^n)}_l'(s)=\widehat{(\partial_sg_\pm^n)}_l(s)$ are the Fourier coefficients of
$g_\pm^n=g_\pm\left(z_{+,N}^n,z_{-,N}^n\right)$,
$w^n=w\left(z_{+,N}^n,z_{-,N}^n,r_N^n;s\right)$,
$\partial_sf_\pm^n=2\lambda z_\pm^n\; {\rm Re} \left[\overline{z_\pm^n}\partial_s z_\pm^n+
2\overline{z_\mp^n}\partial_s z_\mp^n\right]+\lambda\partial_s z_\pm^n
\left[|z_\pm^n|^2+2|z_\mp^n|^2\right]=
:\dot{f}_\pm\left(z_+^n,z_-^n;\partial_sz_+^n,\partial_sz_-^n\right)$ and
$\partial_sg_\pm^n=2\lambda z_\pm^nz_\mp^n\partial_sz_\pm^n+\lambda \left(z_\pm^n\right)^2\partial_sz_\mp^n
=:\dot{g}_\pm\left(z_+^n,z_-^n;\partial_sz_+^n,\partial_sz_-^n\right)$,
respectively,
and (their detailed explicit formulas are shown in \cite[appendix]{Zhao})
\begin{equation}\label{cldl765}
\left\{\begin{split}
&c_l(\tau)=\int_0^{\tau}b_l(\tau-\theta)\, d\theta,\qquad p_l(\tau)=\int_0^\tau\frac{\sin\left(\omega_l(\tau-\theta)\right)}{\eps^2\omega_l}\fe^{3i\theta/\eps^2}\,d\theta,\\
&d_l(\tau)=\int_0^{\tau}b_l(\tau-\theta)\theta\, d\theta,\qquad
q_l(\tau)=\int_0^\tau\frac{\sin\left(\omega_l(\tau-\theta)\right)}{\eps^2\omega_l}\fe^{3i\theta/\eps^2}\theta\,d\theta,\\
&c_l'(\tau)=\int_0^{\tau}b_l'(\tau-\theta)\, d\theta,\qquad p_l'(\tau)=\int_0^\tau\frac{\cos\left(\omega_l(\tau-\theta)\right)}{\eps^2}\fe^{3i\theta/\eps^2}\,d\theta,\\
&d_l'(\tau)=\int_0^{\tau}b_l'(\tau-\theta)\theta\, d\theta,\qquad
q_l'(\tau)=\int_0^\tau\frac{\cos\left(\omega_l(\tau-\theta)\right)}{\eps^2}\fe^{3i\theta/\eps^2}\theta\,d\theta.
\end{split}\right.
\end{equation}
Inserting (\ref{z r app}) into (\ref{z r PN}) and its time derivative with setting $s=\tau$, and
noticing (\ref{u n+1}), we immediately obtain a {\sl  MTI-FP} discretization for the problem (\ref{KG trun}).

In practice, the integrals for computing the Fourier transform
coefficients in (\ref{Fourier tans}), (\ref{VCF}) and (\ref{dz dr})
are usually approximated by numerical
 quadratures \cite{Shen,Dong,Cai2}.  Let $u_j^n$ and $\dot{u}_j^n$ be approximations of $u(x_j,t_n)$ and $\partial_tu(x_j,t_n)$,
 respectively; and $z_{\pm,j}^{n+1},$ $\dot{z}_{\pm,j}^{n+1},$ $r_j^{n+1}$ and $\dot{r}_j^{n+1}$ be approximations of $z_\pm^n(x_j,\tau),$ $\partial_sz_\pm^n(x_j,\tau),$ $r^n(x_j,\tau)$ and $\partial_s r^n(x_j,\tau)$, respectively,
 for $j=0,\ldots,N$. Choosing $u_j^0=\phi_1(x_j)$ and $\dot{u}_j^0=\phi_2(x_j)/\eps^2$ for
 $0\le j\le N$ and noticing (\ref{u n+1}), (\ref{z r PN}) with $s=\tau$, (\ref{z r app}),
 (\ref{MDF ini}) and (\ref{Initial}),
 then a MTI-FP discretization for the problem (\ref{KG trun}) reads  for $n\ge0$
\begin{equation}\label{MTI-FP S}
\left\{
\begin{split}
&u^{n+1}_j=\fe^{i\tau/\eps^2}z_{+,j}^{n+1}+\fe^{-i\tau/\eps^2}
\overline{z_{-,j}^{n+1}}+r^{n+1}_j,\qquad j=0,1,\ldots,N,\\
&\dot{u}^{n+1}_j=\fe^{i\tau/\eps^2}\left(\dot{z}_{+,j}^{n+1}+\frac{i}{\eps^2}z_{+,j}^{n+1}\right)
+\fe^{-i\tau/\eps^2}\left(\overline{\dot{z}_{-,j}^{n+1}}-\frac{i}{\eps^2}
\overline{z_{-,j}^{n+1}}\right)+\dot{r}^{n+1}_j,
\end{split}
\right.
\end{equation}
where
\begin{equation}\label{zpm4567}
\left\{
\begin{split}
&z_{\pm,j}^{n+1}=\sum_{l=-N/2}^{N/2-1}\widetilde{(z_{\pm}^{n+1})}_l\fe^{i\mu_l(x_j-a)},\qquad
r_{j}^{n+1}=\sum_{l=-N/2}^{N/2-1}\widetilde{(r^{n+1})}_l\fe^{i\mu_l(x_j-a)},\\
&\dot{z}_{\pm,j}^{n+1}=\sum_{l=-N/2}^{N/2-1}\widetilde{(\dot{z}_{\pm}^{n+1})}_l\fe^{i\mu_l(x_j-a)},\qquad
\dot{r}_{j}^{n+1}=\sum_{l=-N/2}^{N/2-1}\widetilde{(\dot{r}^{n+1})}_l\fe^{i\mu_l(x_j-a)},
\end{split}
\right.
\end{equation}
with
\begin{equation} \label{Fourier psu coeff}
\left\{
\begin{split}
&\widetilde{(z_\pm^{n+1})}_l=a_l(\tau)\widetilde{(z_\pm^{0})}_l+\eps^2b_l(\tau)\widetilde{(\dot{z}_\pm^{0})}_l
-c_l(\tau)\widetilde{(f_\pm^0)}_l-d_l(\tau)\widetilde{(\dot{f}_\pm^0)}_l,\\
&\widetilde{(\dot{z}_\pm^{n+1})}_l= a_l'(\tau)\widetilde{(z_\pm^0)}_l+\eps^2b_l'(\tau)\widetilde{(\dot{z}_\pm^0)}_l
-c_l'(\tau)\widetilde{(f_\pm^0)}_l-d_l'(\tau)\widetilde{(\dot{f}_\pm^0)}_l,\\
&\widetilde{(r^{n+1})}_l=\frac{\sin(\omega_l \tau)}{\omega_l}\widetilde{(\dot{r}^{0})}_l
-p_l(\tau)\widetilde{(g_+^0)}_l-q_l(\tau)\widetilde{(\dot{g}_+^0)}_l-\overline{p_l}(\tau)
\widetilde{\left(\overline{g_-^0}\right)}_l
-\overline{q_l}(\tau)\widetilde{\left(\overline{\dot{g}_-^0}\right)}_l,\\
&\widetilde{(\dot{r}^{n+1})}_l=\cos(\omega_l \tau)\widetilde{(\dot{r}^0)}_l-p_l'(\tau)\widetilde{(g_+^0)}-q_l'(\tau)\widetilde{(\dot{g}_+^0)}_l
-\overline{p_l'}(\tau)\widetilde{\left(\overline{g_-^0}\right)}_l\\
&\qquad \qquad -\overline{q_l'}(\tau)
\widetilde{\left(\overline{\dot{g}_-^0}\right)}_l -\frac{\tau}{2\eps^2}\widetilde{(w^{n+1})}_l,
\quad -\frac{N}{2}\le l\le \frac{N}{2}-1,
\end{split}
\right.
\end{equation}
and
\begin{equation} \label{MTI-FP E}
\left\{
\begin{split}
&\widetilde{(z_{+}^{0})}_l=\frac{1}{2}\left[\widetilde{(u^n)}_l-i\eps^2\widetilde{(\dot{u}^n)}_l\right],\quad \widetilde{(z_{-}^{0})}_l=\frac{1}{2}\left[\widetilde{(\overline{u^n})}_l-i\eps^2\widetilde{\left(\overline{\dot{u}^n}\right)}_l\right],\\
&\widetilde{(\dot{z}_{\pm}^0)}_l=\frac{i}{2}\left[\frac{2}{\tau}\sin\left(\frac{1}{2}\mu_l^2\tau\right)\widetilde{(z_\pm^0)}_l+\widetilde{(f_{\pm}^0)}_l\right],\quad \widetilde{(\dot{r}^0)}_l=-\widetilde{(\dot{z}_{+}^0)}_l-\widetilde{\left(\overline{\dot{z}_{-}^0}\right)}_l;\\
&f_{\pm,j}^0=f_\pm\left(z_{+,j}^0,z_{-,j}^0\right),\quad \dot{f}_{\pm,j}^0=\dot{f}_{\pm}\left(z_{+,j}^0,z_{-,j}^0;\dot{z}_{+,j}^0,\dot{z}_{-,j}^0\right),\\
&g_{\pm,j}^0=g_\pm\left(z_{+,j}^0,z_{-,j}^0\right),\quad \dot{g}_{\pm,j}^0=\dot{g}_\pm\left(z_{+,j}^0,z_{-,j}^0;\dot{z}_{+,j}^0,\dot{z}_{-,j}^0\right),\\
&w_j^{n+1}=f\left(u_j^{n+1}\right)-f\left(\fe^{i\tau/\eps^2}z_{+,j}^{n+1}
+\fe^{-i\tau/\eps^2}\overline{z_{-,j}^{n+1}}\right),\qquad 0
\le j\le N.
\end{split}
\right.
\end{equation}
This MTI-FP method for the KG equation (\ref{KG trun}) (or (\ref{KG})) is explicit, accurate,
easy to implement and very efficient due to the fast Fourier transform (FFT), and its memory
cost is $O(N)$ and the computational cost per time step is $O(N \log N )$.
\begin{remark}
Instead of discretizing the initial velocity $\partial_sz_\pm^n(x,0)$ from (\ref{MDF ini}) in Fourier space as $\widetilde{(z_\pm^n)}_l'(0)=\frac{i}{2}[\mu_l^2\widetilde{(z_\pm^n)}_l(0)+\widetilde{(f_{\pm}^n)}_l(0)]$ which will result a second order decreasing in the spatial accuracy, we change to the modified coefficients given in (\ref{MTI-FP E}) as filters where the accuracy is now controlled by the time step $\tau$ (cf. \eqref{eq:iv-err} ).
There are other possible choices of the filters.
\end{remark}

\begin{remark}
When the initial data $\phi_1(x)$ and $\phi_2(x)$ are real-valued functions
and $f(u): \bR\to\bR$ in (\ref{KG}), then the solution $u(x,t)$ is real-valued.
In this case, for $n\ge0$, it is easy to see that $z_+^n(x,s)=z_-^n(x,s)$ for $x\in\Omega$ and
 $0\le s\le \tau$ in the MDF (\ref{MDF}). In the corresponding numerical scheme, we have
$z_{+,j}^n=z_{-,j}^n$ for $j=0,\ldots,N$ in the MTI-FP (\ref{MTI-FP S}). Thus the scheme can be simplified and the computational cost can be reduced.
\end{remark}

\section{Uniform convergence of MTI-FP} \label{sec: convergence}
In this section, we establish an error bound for the MTI-FP (\ref{MTI-FP S}) of the problem (\ref{KG trun}),
 which is uniformly for $\eps\in(0,1]$.  Let $0<T< T^*$ with $T^*$ the maximum
existence time of the solution $u$ to the problem (\ref{KG trun}), motivated by the analytical results in \cite{Machihara1,Machihara2,Masmoudi}, here we make the following assumption on the
solution $u$ to the problem (\ref{KG trun}) --- there exists an integer $m_0\ge2$ such that
\[({\rm A}) \quad u\in C^1\left([0,T];H_p^{m_0+4}(\Omega)\right),
\quad \left\|u\right\|_{L^\infty([0,T]; H^{m_0+4})}+\eps^2
\left\|\partial_t u\right\|_{L^\infty([0,T]; H^{m_0+4})}\lesssim 1,
\]
where $H_p^{m}(\Omega)=\left\{\phi(x)\in H^m(\Omega)\ |\ \phi^{(k)}(a)=\phi^{(k)}(b),\
k=0,1,\ldots,m-1\right\}\subset H^m(\Omega)$. Denote
\begin{equation}
C_0=\max_{0<\eps\le 1}\left\{\left\|u\right\|_{L^\infty([0,T];H^{m_0+4})},
\ \eps^2\left\|\partial_tu\right\|_{L^\infty([0,T];H^{m_0+4})}\right\}.
\end{equation}

Let $\bu^n=(u_0^n,u_1^n,\ldots,u_{N}^n)\in {\mathbb C}^{N+1}$,
$\dot{\bu}^n=(\dot{u}_0^n,\dot{u}_1^n,\ldots,\dot{u}_{N}^n)\in {\mathbb C}^{N+1}$ ($n\ge0$)
be the numerical solution obtained from the MTI-FP method (\ref{MTI-FP S}), denote their
interpolations as
\begin{equation} \label{uI854}
u_I^n(x):=(I_N\bu^n)(x),\qquad \dot{u}_I^n(x):=(I_N\dot{\bu}^n)(x),\qquad x\in\overline{\Omega},
\end{equation}
and define the error functions as
\begin{equation}\label{error fun}
e^n(x):=u(x,t_n)-u_I^n(x),\quad \dot{e}^n(x):=\partial_tu(x,t_n)-\dot{u}_I^n(x),\quad x\in\overline{\Omega},\ 0\leq n\leq\frac{T}{\tau},
\end{equation}
then we have the following error estimates for the MTI-FP method (\ref{MTI-FP S}).

\begin{theorem}[Error bounds of MTI-FP]\label{main thm}
Under the assumption (A), there exist two constants $0<h_0\leq1$ and $0<\tau_0\leq1$
sufficiently small and independent of $\eps$ such that, for any $0<\eps\leq1$, when
 $0<h\leq h_0$ and $0<\tau\leq\tau_0$, we have
\begin{align}\label{MTI error bound1}
&\left\|e^n\right\|_{H^2}+\eps^2\left\|\dot{e}^n\right\|_{H^2}\lesssim h^{m_0}+\frac{\tau^2}{\eps^2},
\qquad \left\|e^n\right\|_{H^2}+\eps^2\left\|\dot{e}^n\right\|_{H^2}\lesssim h^{m_0}+\tau^2+\eps^2, \\
&\left\|u^n_I\right\|_{H^2}\leq C_0+1,\quad \left\|\dot{u}^n_I\right\|_{H^2}\leq \frac{C_0+1}{\eps^2},\qquad 0\leq n\leq\frac{T}{\tau}.
\label{MTI sol bound}
\end{align}
Thus, by taking the minimum of the two error bounds in (\ref{MTI error bound1}) for $\eps\in(0,1]$,
we obtain an  error bound which is uniformly convergent for $\eps\in(0,1]$
\begin{equation}
\left\|e^n\right\|_{H^2}+\eps^2\left\|\dot{e}^n\right\|_{H^2}\lesssim h^{m_0}+\tau^2+\min_{0<\eps\leq1}\left\{\frac{\tau^2}{\eps^2},\eps^2\right\}\lesssim h^{m_0}+\tau,
\quad 0\leq n\leq\frac{T}{\tau}.
\end{equation}
\end{theorem}

In order to prove the above theorem, for $0\leq n\leq\frac{T}{\tau}$, we introduce
\begin{equation}
e^{n}_N(x):=(P_Nu)(x,t_n)-u_I^n(x),\ \ \dot{e}^{n}_N(x):=P_N(\partial_{t}u)(x,t_n)-\dot{u}_I^n(x),\ \ \,x\in\overline{\Omega}.\label{e_N}
\end{equation}
Using the triangle inequality and noticing the assumption (A), we have
\begin{subequations}\label{proof: eq0}
\begin{align}
&\|e^{n}\|_{H^2}\leq\|u(\cdot,t_{n})-P_Nu(\cdot,t_{n})\|_{H^2}+
\|e^{n}_N\|_{H^2}\lesssim h^{m_0+2}+\|e^{n}_N\|_{H^2},\\
&\|\dot{e}^{n}\|_{H^2}\leq\|\partial_tu(\cdot,t_{n})-P_N\partial_tu(\cdot,t_{n})\|_{H^2}+
\|\dot{e}^{n}_N\|_{H^2}\lesssim \frac{1}{\eps^2}h^{m_0+2}+\|\dot{e}^{n}_N\|_{H^2}.
\end{align}
\end{subequations}
Thus we need only obtain estimates for $\|e^{n}_N\|_{H^2}$ and $\|\dot{e}^{n}_N\|_{H^2}$, which
will be done by introducing the following error energy functional
\begin{equation}
\mathcal{E}\left(e_N^n,\dot{e}_N^n\right):=\eps^2\left\|\dot{e}_N^n\right\|_{H^2}^2
+\left\|\partial_xe_N^n\right\|_{H^2}^2+\frac{1}{\eps^2}\left\|e_N^n\right\|_{H^2}^2,
\qquad 0\leq n\leq\frac{T}{\tau},\label{error engery}
\end{equation}
and establishing the following several lemmas.

\begin{lemma}[Formulation of the exact solution]\label{lm: solu}
Denote the Fourier expansion of the exact solution $u(x,t)$ of the problem (\ref{KG trun}) as
\begin{equation}\label{uxt876}
u(x,t)=\sum_{l=-\infty}^\infty \widehat{u}_l(t)\; e^{i\mu_l (x-a)},
\qquad x\in \bar{\Omega}, \quad t\ge0,
\end{equation}
then we have
\begin{subequations}\label{subeq1}
\begin{eqnarray} \label{ultnp123}
\widehat{u}_l(t_{n+1})&=&\cos(\omega_l\tau)\widehat{u}_l(t_n)+
\frac{\sin(\omega_l\tau)}{\omega_l}\widehat{u}_l'(t_n)
-\int_0^\tau\frac{\sin(\omega_l(\tau-\theta))}{\eps^2\omega_l}
\Big[\fe^{i\theta/\eps^2}\widehat{(f_+^n)}_l(\theta)\nonumber\\
&&+\fe^{-i\theta/\eps^2}\widehat{\left(\overline{f_-^n}\right)}_l(\theta)
+\fe^{3i\theta/\eps^2}\widehat{(g_+^n)}_l(\theta)
+\fe^{-3i\theta/\eps^2}\widehat{\left(\overline{g_-^n}\right)}_l(\theta)
+\widehat{(w^n)}_l(\theta)\Big]\,d\theta,\\
\label{vcf: u cubic}
\widehat{u}_l'(t_{n+1})&=&\cos(\omega_l\tau)\widehat{u}_l'(t_n)-\omega_l\sin(\omega_l\tau)\widehat{u}_l(t_n)
-\int_0^\tau\frac{\cos(\omega_l(\tau-\theta))}{\eps^2}\Big[\fe^{i\theta/
\eps^2}\widehat{(f_+^n)}_l(\theta)\nonumber\\
&&+\fe^{-i\theta/\eps^2}\widehat{\left(\overline{f_-^n}\right)}_l(\theta)
+\fe^{3i\theta/\eps^2}\widehat{(g_+^n)}_l(\theta)
+\fe^{-3i\theta/\eps^2}\widehat{\left(\overline{g_-^n}\right)}_l(\theta)
+\widehat{(w^n)}_l(\theta)\Big]\,d\theta.
\end{eqnarray}
\end{subequations}
\end{lemma}

\begin{proof}
Substituting (\ref{uxt876}) with $t=t_n+s$ into (\ref{KG trun}), we have
\begin{equation}
\eps^2\widehat{u}_l''(t_n+s)+\left(\mu_l^2+\frac{1}{\eps^2}\right)\widehat{u}_l(t_n+s)+
\widehat{f(u)}_l(t_n+s)=0,\ \  s>0.\label{KG fft eq}
\end{equation}
Applying the variation-of-constant formula to (\ref{KG fft eq}) and noticing (\ref{nd pm def}), we get
\begin{eqnarray}\label{ultns34}
\widehat{u}_l(t_n+s)&=&\cos(\omega_ls)\widehat{u}_l(t_n)+\frac{\sin(\omega_ls)}
{\omega_l}\widehat{u}_l'(t_n)\nonumber\\
&&-\int_0^s\frac{\sin(\omega_l(s-\theta))}{\eps^2\omega_l}\widehat{f(u)}_l(t_n+\theta)\,d\theta,
\quad 0\le s\le \tau.
\end{eqnarray}
For the cubic nonlinearity $f(u)=\lambda |u|^2u$ and noticing (\ref{ansatz}), (\ref{fr cubic})
and (\ref{g_pm cubic}), we have
\begin{eqnarray}\label{fuxts}
f(u(x,t_n+s))&=&\fe^{is/\eps^2}f_+^n(x,s)+\fe^{-is/\eps^2}\overline{f_-^n}(x,s)+
\fe^{3is/\eps^2}g_+^n(x,s)\nonumber\\
&&+\fe^{-3is/\eps^2}\overline{g_-^n}(x,s)+w^n(x,s), \qquad x\in \bar{\Omega}, \quad 0\le s\le \tau,
\end{eqnarray}
where
\begin{equation}\label{fg_pm w}
\left\{\begin{split}
&f_\pm^n(x,s)=f_\pm(z_+^n(x,s),z_-^n(x,s)),\qquad g_\pm^n(x,s)=g_\pm(z_+^n(x,s),z_-^n(x,s)),\\
&w^n(x,s)=w^n(z_+^n(x,s),z_-^n(x,s),r^n(x,s);s),\qquad x\in\overline{\Omega},\ \ 0\leq s\leq\tau.
\end{split}\right.
\end{equation}
Plugging (\ref{fuxts}) and (\ref{fg_pm w}) into (\ref{ultns34}), we get
\begin{align}
\widehat{u}_l(t_n+s)=&\cos(\omega_ls)\widehat{u}_l(t_n)+\frac{\sin(\omega_ls)}{\omega_l}\widehat{u}_l'(t_n)
-\int_0^s\frac{\sin(\omega_l(s-\theta))}{\eps^2\omega_l}
\Big[\fe^{i\theta/\eps^2}\widehat{(f_+^n)}_l(\theta)\nonumber\\
&+\fe^{-i\theta/\eps^2}\widehat{\left(\overline{f_-^n}\right)}_l(\theta)
+\fe^{3i\theta/\eps^2}\widehat{(g_+^n)}_l(\theta)
+\fe^{-3i\theta/\eps^2}\widehat{\left(\overline{g_-^n}\right)}_l(\theta)+
\widehat{(w^n)}_l(\theta)\Big]\,d\theta.\label{vcf: u}
\end{align}
Then
we can obtain (\ref{ultnp123}) by setting  $s=\tau$ in (\ref{vcf: u}) and
get (\ref{vcf: u cubic}) by taking derivative with respect to $s$ in (\ref{vcf: u}) and then letting $s=\tau$.
\end{proof}

\begin{lemma}[A new formulation of MTI-FP]\label{lm: rewrite scheme}
For $n\ge0$, expanding $u_I^n(x)$ and $\dot{u}_I^n(x)$ in (\ref{uI854}) into Fourier series as
\begin{equation} \label{uInx876}
u_I^{n}(x)=\sum_{l=-N/2}^{N/2-1}\widetilde{(u_I^{n})}_l\;\fe^{i\mu_l(x-a)},\quad
\dot{u}_I^{n}(x)=\sum_{l=-N/2}^{N/2-1}\widetilde{(\dot{u}_I^{n})}_l\;\fe^{i\mu_l(x-a)},\qquad
x\in\bar{\Omega},
\end{equation}
then we have
\begin{equation}\label{observ of scheme}
\left\{
\begin{split}
&\widetilde{(u^{n+1}_I)}_l=\cos(\omega_l\tau)\widetilde{(u^{n}_I)}_l
+\frac{\sin(\omega_l\tau)}{\omega_l}\widetilde{(\dot{u}^{n}_I)}_l-\widetilde{G^n_l},\\
&\widetilde{(\dot{u}^{n+1}_I)}_l=-\omega_l\sin(\omega_l\tau)\widetilde{(u^{n}_I)}_l
+\cos(\omega_l\tau)\widetilde{(\dot{u}^{n}_I)}_l-
\widetilde{\dot{G}^n_l},\quad l=-\frac{N}{2},\ldots, \frac{N}{2}-1,
\end{split}\right.
\end{equation}
where
\begin{eqnarray}\label{nonlin mum dt}
\widetilde{G^n_l}&=&\fe^{i\tau/\eps^2}\left[c_l(\tau)\widetilde{(f_+^0)}_l+d_l(\tau)
\widetilde{(\dot{f}_+^0)}_l\right]+\fe^{-i\tau/\eps^2}
\left[\overline{c_l}(\tau)\widetilde{\left(\overline{f_-^0}\right)}_l+\overline{d_l}
(\tau)\widetilde{\left(\overline{\dot{f}_-^0}\right)}_l\right]\nonumber\\
&&+p_l(\tau)\widetilde{(g_+^0)}_l
+q_l(\tau)\widetilde{(\dot{g}_+^0)}_l+\overline{p_l}(\tau)\widetilde{\left(\overline{g_-^0}\right)}_l
+\overline{q_l}(\tau)\widetilde{\left(\overline{\dot{g}_-^0}\right)}_l,
\end{eqnarray}
\begin{eqnarray}\label{nonlin num dt}
\widetilde{\dot{G}^n_l}&=&\fe^{i\tau/\eps^2}\left[c_l'(\tau)+\frac{i}{\eps^2}c_l(\tau)\right]
\widetilde{(f_+^0)}_l+\fe^{i\tau/\eps^2}
\left[d_l'(\tau)+\frac{i}{\eps^2}d_l(\tau)\right]\widetilde{(\dot{f}_+^0)}_l\nonumber\\
&&+\fe^{-i\tau/\eps^2}\left[\overline{c_l'}(\tau)
-\frac{i}{\eps^2}\overline{c_l}(\tau)\right]\widetilde{\left(\overline{f_-^0}\right)}_l+
\fe^{-i\tau/\eps^2}\left[\overline{d_l'}(\tau)
-\frac{i}{\eps^2}\overline{d_l}(\tau)\right]\widetilde{\left(\overline{\dot{f}_-^0}\right)}_l\nonumber\\
&&+p_l'(\tau)\widetilde{(g_+^0)}_l+q_l'(\tau)\widetilde{(\dot{g}_+^0)}_l
+\overline{p_l'}(\tau)\widetilde{\left(\overline{g_-^0}\right)}_l+\overline{q_l'}(\tau)
\widetilde{\left(\overline{\dot{g}_-^0}\right)}_l
+\frac{\tau}{2\eps^2}\widetilde{(w^{n+1})}_l.
\end{eqnarray}

\end{lemma}

\begin{proof} Combining (\ref{MTI-FP E}), (\ref{uInx876}) and (\ref{uI854}), we have
\begin{equation}\label{lm: rewrite scheme eq1}
\left\{\begin{split}
&\widetilde{(z_+^0)}_l=\frac{1}{2}\left[\widetilde{(u^n_I)}_l-i\eps^2\widetilde{(\dot{u}^n_I)}_l\right],\qquad
\widetilde{(z_-^0)}_l=\frac{1}{2}\left[\widetilde{\left(\overline{u^n_I}\right)}_l-i\eps^2
\widetilde{\left(\overline{\dot{u}^n_I}\right)}_l\right],\\
&\widetilde{(\dot{r}^0)}_l=-\widetilde{(\dot{z}_+^0)}_l-\widetilde{\left(\overline{\dot{z}_-^0}\right)}_l,\qquad l=-\frac{N}{2},\ldots,\frac{N}{2}-1.
\end{split}\right.
\end{equation}
Inserting (\ref{lm: rewrite scheme eq1}) into (\ref{Fourier psu coeff}) and noticing
(\ref{MTI-FP S}), (\ref{zpm4567}), (\ref{uInx876}) and (\ref{uI854}), we get
\begin{subequations}\label{subeq3}
\begin{eqnarray}
\widetilde{(u^{n+1}_I)}_l&=&\fe^{i\tau/\eps^2}\widetilde{(z_+^{n+1})}_l
+\fe^{-i\tau/\eps^2}\widetilde{\left(\overline{z_-^{n+1}}\right)}_l
+\widetilde{(r^{n+1})}_l\nonumber\\
&=&{\rm Re}\left\{\fe^{i\tau/\eps^2}a_l(\tau)\right\}\widetilde{(u^{n}_I)}_l
+\eps^2\;{\rm Im}\left\{\fe^{i\tau/\eps^2}a_l(\tau)\right\}\widetilde{(\dot{u}^{n}_I)}_l
+\eps^2\fe^{i\tau/\eps^2}b_l(\tau)\widetilde{(\dot{z}_+^0)}_l\nonumber\\
&&+\eps^2\fe^{-i\tau/\eps^2}\overline{b_l}(\tau)\widetilde{\left(\overline{\dot{z}_-^0}\right)}_l+
\frac{\sin(\omega_l\tau)}{\omega_l}\widetilde{(\dot{r}^0)}_l-\widetilde{G^n_l},
\quad l=-\frac{N}{2},\ldots,\frac{N}{2}-1,\label{lm: rewrite scheme eq0}\\
\widetilde{(\dot{u}^{n+1}_I)}_l&=&\fe^{i\tau/\eps^2}\Big[\widetilde{(\dot{z}_+^{n+1})}_l
+\frac{i}{\eps^2}\widetilde{(z_+^{n+1})}_l \Big]
+\fe^{-i\tau/\eps^2}\Big[\widetilde{\left(\overline{\dot{z}_-^{n+1}}\right)}_l
-\frac{i}{\eps^2}\widetilde{\left(\overline{z_-^{n+1}}\right)}_l \Big]+\widetilde{(\dot{r}^{n+1})}_l\nonumber\\
&=&{\rm Re}\left\{\fe^{i\tau/\eps^2}a_l'(\tau)
+\frac{i}{\eps^2}\fe^{\frac{i\tau}{\eps^2}}a_l(\tau)\right\}\widetilde{(u^{n}_I)}_l
+\eps^2\fe^{-i\tau/\eps^2}\left[\overline{b_l'}(\tau)-\frac{i}{\eps^2}\overline{b_l}(\tau)\right]
\widetilde{\left(\overline{\dot{z}_-^0}\right)}_l\nonumber\\
&&+\eps^2\;{\rm Im}\left\{\fe^{i\tau/\eps^2}a_l'(\tau)
+\frac{i}{\eps^2}\fe^{\frac{i\tau}{\eps^2}}a_l(\tau)\right\}\widetilde{(\dot{u}^{n}_I)}_l
+\eps^2\fe^{i\tau/\eps^2}\left[b_l'(\tau)+\frac{i}{\eps^2}b_l(\tau)\right]\widetilde{(\dot{z}_+^0)}_l\nonumber\\
&&+\cos(\omega_l\tau)\widetilde{(\dot{r}^0)}_l-\widetilde{\dot{G}^n_l},
\qquad l=-\frac{N}{2},\ldots,\frac{N}{2}-1,\label{unp134567}
\end{eqnarray}
\end{subequations}
where ${\rm Re}(\alpha)$ and ${\rm Im}(\alpha)$ denote the real and imaginary parts of a complex number
$\alpha$, respectively. Thus we can obtain (\ref{observ of scheme}) from
(\ref{subeq3}) by using the fact
that $a_l(\tau)=a_{-l}(\tau)$ and  $b_l(\tau)=b_{-l}(\tau)$ for $l=-N/2,\ldots,N/2-1$
in (\ref{nd pm def}).
\end{proof}

For $0\leq n\leq\frac{T}{\tau}$, let ${z}_\pm^n(x,s)$ and $r^n(x,s)$ be the solution of the MDF (\ref{MDF trun})-(\ref{MDF ini})
with $\phi_1^n(x)=u(x,t_n)$ and $\phi_2^n(x)=\eps^2 \partial_t u(x,t_n)$ for $x\in\bar{\Omega}$,
then we have

\begin{lemma}[A prior estimate of MDF] \label{lm: prior}
Under the assumption (A), there exists a constant $\tau_1>0$
independent of $0<\eps\le1$ and $h>0$, such that for $0<\tau\leq\tau_1$
\begin{align}
&\left\|z_\pm^n\right\|_{L^\infty([0,\tau];H^{m_0+2})}+\left\|\partial_sz_\pm^n\right
\|_{L^\infty([0,\tau];H^{m_0+1})}+
\left\|\partial_{ss}z_\pm^n\right\|_{L^\infty([0,\tau];H^{m_0})}\lesssim1,\label{reg z}\\
&\left\|r^n\right\|_{L^\infty([0,\tau];H^{4})}+\eps^2\left\|\partial_s r^n\right\|_{L^\infty([0,\tau];H^{3})}+\eps^4\left\|\partial_{ss}
r^n\right\|_{L^\infty([0,\tau];H^{2})}\lesssim \eps^2. \label{reg r}
\end{align}
\end{lemma}

\begin{proof}
From (\ref{MDF ini}) and noticing the assumption (A) and (\ref{fr cubic}), we have
\begin{eqnarray*}
&&\|z_\pm^n(\cdot,0)\|_{H^{m_0+4}}\lesssim \|u(\cdot,t_n)\|_{H^{m_0+4}}+\eps^2\|\partial_tu(\cdot,t_n)\|_{H^{m_0+4}}\lesssim1,\\
&&\|\partial_sz_\pm^n(\cdot,0)\|_{H^{m_0+2}}\lesssim\|\partial_{xx}z_\pm^n(\cdot,0)
\|_{H^{m_0+2}}+\|f_\pm(z_+^n(\cdot,0),z_-^n(\cdot,0))\|_{H^{m_0+2}}
\lesssim1,
\end{eqnarray*}
which immediately imply
\begin{equation}
\|\partial_sr^n(\cdot,0)\|_{H^{m_0+2}}\leq\|\partial_sz_+^n(\cdot,0)\|_{H^{m_0+2}}
+\|\partial_sz_-^n(\cdot,0)\|_{H^{m_0+2}}\lesssim1. \label{lm: prior eq5}
\end{equation}
Similar to the proof for the nonlinear Schr\"{o}dinger equation with wave operator \cite{Cai1,Cai2},
we can easily establish (\ref{reg z})  and the details are omitted here for brevity. Taking the Fourier expansion of
$r^n(x,s)$ and noticing (\ref{MDF trun}), (\ref{MDF ini}), (\ref{fr cubic}) and (\ref{g_pm cubic}), we obtain
\begin{equation}
r^n(x,s)=\sum_{l=-\infty}^\infty \widehat{(r^n)}_l(s)\; e^{i\mu_l(x-a)}, \qquad x\in \bar{\Omega}, \quad
0\le s\le \tau,
\end{equation}
where  for $l\in {\mathbb Z}$
\begin{align}\label{r l cubic}
\widehat{(r^n)}_l(s)=&\frac{\sin(\omega_l s)}{\omega_l}\widehat{(r^n)}_l'(0)
-\int_0^s\frac{\sin\left(\omega_l(s-\theta)\right)}{\eps^2\omega_l}\fe^{3i\theta/\eps^2}
\widehat{(g_+^n)}_l(\theta)\,d\theta \nonumber\\
&-\int_0^s\frac{\sin\left(\omega_l(s-\theta)\right)}{\eps^2\omega_l}\fe^{-3i\theta/\eps^2}
\widehat{\left(\overline{g_-^n}\right)}_l(\theta)\,d\theta
-\int_0^s\frac{\sin\left(\omega_l(s-\theta)\right)}{\eps^2\omega_l}\widehat{(w^n)}_l(\theta)\,d\theta.
\end{align}
Let $N_\eps:=\left[\frac{b-a}{2\pi\eps}\right]=O\left(\frac{1}{\eps}\right)$ be the integer part of $\frac{b-a}{2\pi\eps}$.
From (\ref{r l cubic}),  integrating by parts and using the Cauchy's
and H\"{o}lder's inequalities, we obtain for $|l|\le N_\eps$
\begin{eqnarray}\label{lm: prior eq2}
|\widehat{(r^n)}_l(s)|^2&\lesssim&\biggl|\eps^2|\widehat{(r^n)}_l'(0)|
+\eps^2\left[|\widehat{(g_+^n)}_l(s)|+|\widehat{(g_+^n)}_l(0)|+|\widehat{(\overline{g_-^n})}_l(s)|
+|\widehat{(\overline{g_-^n})}_l(0)|\right]\nonumber\\
&&+\int_0^s\left[\eps^2\left(|\widehat{(g_+^n)}_l'(\theta)|+|\widehat{(\overline{g_-^n})}_l'(\theta)|\right)
+|\widehat{(w^n)}_l(\theta)|\right] d\theta\biggr|^2\nonumber\\
&\lesssim&\eps^4\left[|\widehat{(r^n)}_l'(0)|^2
+|\widehat{(g_+^n)}_l(s)|^2+|\widehat{(g_+^n)}_l(0)|^2+|\widehat{(\overline{g_-^n})}_l(s)|^2
+|\widehat{(\overline{g_-^n})}_l(0)|^2\right]\nonumber\\
&&+\int_0^s\left[\eps^4\left(|\widehat{(g_+^n)}_l'(\theta)|^2+|\widehat{(\overline{g_-^n})}_l'(\theta)|^2\right)
+|\widehat{(w^n)}_l(\theta)|^2\right] d\theta.
\end{eqnarray}
Here we use the fact that for $|l|\le N_\eps$
\begin{eqnarray*}
T_l(\theta)&=&\frac{\eps^2\fe^{3i\theta/\eps^2}}{\eps^4\omega_l^2-9}\left[\cos(\omega_l(s-\theta))
+\frac{3i}{\eps^2\omega_l}\sin(\omega_l(s-\theta))\right]=O(\eps^2),\\
T_l^\prime(\theta)&=&\frac{\sin(\omega_l(s-\theta))}{\eps^2\omega_l}
\fe^{3i\theta/\eps^2}=O(1), \quad 0\leq\theta\leq s\leq \tau, \qquad 0<\eps\le 1.
\end{eqnarray*}
Similarly, we can get for $|l|> N_\eps$
\begin{equation}\label{rn9765}
|\widehat{(r^n)}_l(s)|^2\lesssim \eps^4|\widehat{(r^n)}_l'(0)|^2+\int_0^s\left[|\widehat{(g_+^n)}_l(\theta)|^2
+|\widehat{(\overline{g_-^n})}_l(\theta)|^2+|\widehat{(w^n)}_l(\theta)|^2\right] d\theta.
\end{equation}
Multiplying (\ref{lm: prior eq2}) and (\ref{rn9765}) by $1+\mu_l^2+\ldots+\mu_l^8$, then
summing them up for $l\in{\mathbb Z}$, we obtain
\begin{eqnarray}\label{rn987}
&&\|r^n(\cdot,s)\|_{H^{4}}^2\lesssim\sum_{l=-\infty}^\infty \left(1+\mu_l^2+\ldots+\mu_l^8\right)|\widehat{(r^n)}_l(s)|^2\nonumber\\
&&\qquad\lesssim\sum_{l=-\infty}^\infty\left(\sum_{m=0}^4
\mu_l^{2m}\right)\int_0^s|\widehat{(w^n)}_l(\theta)|^2d\theta
+\eps^4\Big[\|\partial_sr^n(\cdot,0)\|_{H^4}^2 +\|g_+^n\|_{L^\infty([0,\tau];H^4)}\nonumber\\
&&\qquad +\|g_-^n\|_{L^\infty([0,\tau];H^4)}+s\|\partial_sg_+^n\|_{L^\infty([0,\tau];H^4)}
+s\|\partial_sg_-^n\|_{L^\infty([0,\tau];H^4)}\Big] \nonumber\\
&&\qquad +s\Big[\|g_+^n-P_{N_\eps}g_+^n\|_{L^\infty([0,\tau];H^4)}^2
+\|g_-^n-P_{N_\eps}g_-^n\|_{L^\infty([0,\tau];H^4)}^2\Big]\nonumber\\
&&\qquad\lesssim\eps^4+\int_0^s\|w^n(\cdot,\theta)\|_{H^{4}}^2\,d\theta
\lesssim\eps^4+\int_0^s\|r^n(\cdot,\theta)\|_{H^{4}}^2\,d\theta,\qquad 0\le s\le \tau.
\end{eqnarray}
Combining (\ref{rn987}), (\ref{lm: prior eq5}), noticing $r^n(x,0)\equiv 0$ for $x\in\bar{\Omega}$,
and adapting the standard bootstrap argument for the nonlinear wave equation \cite{Tao}, we have that
there exists a positive constant $\tau_1>0$ independent of $\eps$ and $h$ such that
\begin{equation}\label{rn765}
\left\| r^n\right\|_{L^\infty([0,\tau];H^{4})}\lesssim\eps^2.
\end{equation}
Similarly we can obtain
\begin{equation}
\left\|\partial_s r^n\right\|_{L^\infty([0,\tau];H^{3})}\lesssim1,\qquad
\left\|\partial_{ss} r^n\right\|_{L^\infty([0,\tau];H^{2})}\lesssim\frac{1}{\eps^{2}},
\end{equation}
which, together with (\ref{rn765}), immediately imply the desired inequality (\ref{reg r}).
\end{proof}

Combining the above lemmas and defining the local truncation error as
\begin{equation}
\xi^n(x)=\sum_{l=-N/2}^{N/2-1}\widehat{\xi^n_l}\;e^{i\mu_l (x-a)}, \qquad
\dot{\xi}^n(x)=\sum_{l=-N/2}^{N/2-1}\widehat{\dot{\xi}^n_l}\;e^{i\mu_l (x-a)},
\qquad x\in\bar{\Omega},
\end{equation}
where
\begin{equation}\label{loc error}
\left\{
\begin{split}
&\widehat{\xi^n_l}:=\widehat{u}_l(t_{n+1})-\left[\cos(\omega_l\tau)\widehat{u}_l(t_n)+
\frac{\sin(\omega_l\tau)}{\omega_l}\widehat{u}_l'(t_n)-\widehat{\mathcal{G}_l^n}\right],\\
&\widehat{\dot{\xi}^n_l}:=\widehat{u}_l'(t_{n+1})-\left[-\omega_l\sin(\omega_l\tau)\widehat{u}_l(t_n)+
\cos(\omega_l\tau)\widehat{u}_l'(t_n)-\widehat{\dot{\mathcal{G}}_l^n}\right],
\end{split}\right.
\end{equation}
with
{\small \begin{subequations}\label{subeq4}
\begin{eqnarray}
\widehat{\mathcal{G}_l^n}&=&\fe^{i\tau/\eps^2}\left[c_l(\tau)\widehat{(f_+^n)}_l(0)
+d_l(\tau)\widehat{(f_+^n)}_l'(0)\right]
+\fe^{-i\tau/\eps^2}\Big[\overline{c_l}(\tau)\widehat{\left(\overline{f_-^n}\right)}_l(0)
+\overline{d_l}(\tau)\widehat{\left(\overline{f_-^n}\right)}_l'(0)\Big]\nonumber\\
&&+p_l(\tau)\widehat{(g_+^n)}_l(0)
+q_l(\tau)\widehat{(g_+^n)}_l'(0)+\overline{p_l}(\tau)\widehat{\left(\overline{g_-}\right)}_l(0)
+\overline{q_l}(\tau)\widehat{\left(\overline{g_-}\right)}_l'(0),\label{nonlin exact}\\
\widehat{\mathcal{\dot{G}}_l^n}&=&\fe^{i\tau/\eps^2}\left[c_l'(\tau)
+\frac{i}{\eps^2}c_l(\tau)\right]\widehat{(f_+^n)}_l(0)+\fe^{i\tau/\eps^2}
\left[d_l'(\tau)+\frac{i}{\eps^2}d_l(\tau)\right]\widehat{(f_+^n)}_l'(0)\nonumber\\
&&+\fe^{-i\tau/\eps^2}\left[\overline{c_l'}(\tau)
-\frac{i}{\eps^2}\overline{c_l}(\tau)\right]\widehat{\left(\overline{f_-^n}\right)}_l(0)
+\fe^{-i\tau/\eps^2}\left[\overline{d_l'}(\tau)
-\frac{i}{\eps^2}\overline{d_l}(\tau)\right]\widehat{\left(\overline{f_-^n}\right)}_l'(0)\nonumber\\
&&+p_l'(\tau)\widehat{(g_+^n)}_l(0)+q_l'(\tau)\widehat{(g_+^n)}_l'(0)
+\overline{p_l'}(\tau)\widehat{\left(\overline{g_-^n}\right)}_l(0)+\overline{q_l'}
(\tau)\widehat{\left(\overline{g_-^n}\right)}_l'(0)
+\frac{\tau}{2\eps^2}\widehat{(w^n)}_l(\tau).\label{nonlin exact dt}
\end{eqnarray}
\end{subequations}}
Then we have the following estimates for them.

\begin{lemma}[Estimates on $\xi^n$ and $\dot{\xi}^n$]\label{lm:local_error}
Under the assumption (A), when $0<\tau\le \tau_1$, we have two independent estimates for $0<\eps\le 1$
\begin{equation}
\mcE\left(\xi^{n},\dot{\xi}^{n}\right)\lesssim \frac{\tau^6}{\eps^2}+\tau^2\eps^2 \quad \hbox{and}\quad
\mcE\left(\xi^{n},\dot{\xi}^{n}\right)\lesssim \frac{\tau^6}{\eps^6},\qquad
n=0,1,\ldots,\frac{T}{\tau}-1.
  \label{lm:local_error_L result}
\end{equation}
\end{lemma}

\begin{proof}
Noticing the fact
\begin{equation}\label{fact1 proof}
b_l(\tau-\theta)\fe^{i\tau/\eps^2}=\frac{\sin(\omega_l(\tau-\theta))}{\eps^2\omega_l}
\fe^{i\theta/\eps^2}, \qquad 0\le \theta\le \tau,
\end{equation}
subtracting (\ref{loc error}) from (\ref{subeq1})
and then using the Taylor's expansion,  we get
\begin{eqnarray}\label{lm:local eq0}
\widehat{\xi^n_l}&=&-\int_0^\tau\frac{\sin(\omega_l(\tau-\theta))}{\eps^2\omega_l}
\Big[\theta^2\Big(\fe^{i\theta/\eps^2}\int_0^1\widehat{(f_+^n)}_l''(\theta\rho)(1-\rho)d\rho\nonumber\\
&&\,+\fe^{-i\theta/\eps^2}\int_0^1\widehat{(\overline{f_-^n})}_l''(\theta\rho)(1-\rho)d\rho +\fe^{3i\theta/\eps^2}\int_0^1\widehat{(g_+^n)}_l''(\theta\rho)(1-\rho)d\rho\nonumber\\
&&\,+\fe^{-3i\theta/\eps^2}\int_0^1\widehat{(\overline{g_-^n})}_l''(\theta\rho)(1-
\rho)d\rho\Big)+\widehat{(w^n)}_l(\theta)\Big]
d\theta, \qquad l=-\frac{N}{2},\ldots,\frac{N}{2}-1.
\end{eqnarray}
Using the triangle inequality, we obtain
\begin{align*}
|\widehat{\xi^n_l}|\lesssim&\frac{\tau^2}{\eps^2\omega_l}\int_0^\tau
\bigg[\int_0^1\left|\widehat{(f_+^n)}_l''(\theta\rho)\right|d\rho+
\int_0^1\left|\widehat{(\overline{f_-^n})}_l''(\theta\rho)\right|d\rho+
\int_0^1\left|\widehat{(g_+^n)}_l''(\theta\rho)\right|d\rho\\
&+\int_0^1\left|\widehat{(\overline{g_-^n})}_l''(\theta\rho)\right|d\rho\bigg]d\theta
+\frac{1}{\eps^2\omega_l}\int_0^\tau\left|\widehat{(w^n)}_l(\theta)\right|d\theta,\qquad l=-\frac{N}{2},\ldots,\frac{N}{2}-1.
\end{align*}
Noting $\frac{1}{\eps^2\omega_l}=\frac{1}{\sqrt{1+\eps^2\mu_l^2}}\le 1$ for $l=-\frac{N}{2},\ldots,\frac{N}{2}-1$
and by Lemma \ref{lm: prior}, we get
\begin{align}
\|\xi^n\|_{H^2}^2\lesssim& \tau^6\Big[\|\partial_{ss}f_+^n\|_{L^\infty([0,\tau];H^2)}^2+\|\partial_{ss}f_-^n\|_{L^\infty([0,T];H^2)}^2+
\|\partial_{ss}g_+^n\|_{L^\infty([0,\tau];H^2)}^2\nonumber\\
&+\|\partial_{ss}g_-^n\|_{L^\infty([0,\tau];H^2)}^2\Big]+\tau^2\|w^n\|_{L^\infty([0,\tau];H^2)}^2
\lesssim\tau^6+\tau^2\eps^4,\quad 0<\tau\leq\tau_1.\label{lm:local eq1}
\end{align}
Similarly, noting  $\frac{|\mu_l|}{\eps^2\omega_l}=\frac{|\mu_l|}{\sqrt{1+\eps^2\mu_l^2}}\le\frac{1}{\eps}$
for $l=-\frac{N}{2},\ldots,\frac{N}{2}-1$, we obtain
\begin{align}
\|\partial_x\xi^n\|_{H^2}^2\lesssim\frac{\tau^6}{\eps^2}+\tau^2\eps^2 \quad \hbox{and}\quad \|\dot{\xi}^n\|_{H^2}^2\lesssim\frac{\tau^6}{\eps^4}+\tau^2, \quad 0<\tau\leq\tau_1.
\label{lm:local eq2}
\end{align}
Plugging (\ref{lm:local eq1}) and (\ref{lm:local eq2}) into (\ref{error engery}) with
$e_N^n=\xi^{n}$  and $\dot{e}_N^n=\dot{\xi}^{n}$, we immediately
get the first inequality in (\ref{lm:local_error_L result}).
On the other hand, for $l=-N/2,\ldots,N/2-1$, noticing $\widehat{(w^n)}_l(0)=0$ and
using the error formula of trapezoidal rule for an integral, we get
\begin{equation} \label{int9876}
\left|\int_0^\tau\frac{\sin(\omega_l(\tau-\theta))}{\eps^2\omega_l}
\widehat{(w^n)}_l(\theta)
d\theta\right|\lesssim \int_0^\tau\frac{\theta(\tau-\theta)}{\eps^2\omega_l}\left|\frac{d^2}{d\theta^2}\left[
\sin(\omega_l(\tau-\theta))\widehat{(w^n)}_l(\theta)\right]\right|d\theta.
\end{equation}
Combining (\ref{int9876}) and (\ref{lm:local eq0}), we have
\begin{align*}
|\widehat{\xi^n_l}|\lesssim&\frac{\tau^2}{\eps^2\omega_l}\int_0^\tau\bigg[\int_0^1\left|
\widehat{(f_+^n)}_l''(\theta\rho)\right|d\rho+
\int_0^1\left|\widehat{(\overline{f_-^n})}_l''(\theta\rho)\right|d\rho+\int_0^1\left|
\widehat{(g_+^n)}_l''(\theta\rho)\right|d\rho\\
&+\int_0^1\left|\widehat{(\overline{g_-^n})}_l''(\theta\rho)\right|d\rho+
\omega_l^2\left|\widehat{(w^n)}_l(\theta)\right|
+\omega_l\left|\widehat{(w^n)}_l'(\theta)\right|+\left|\widehat{(w^n)}_l''(\theta)\right|\bigg]d\theta.
\end{align*}
Noting $\omega_l\lesssim(1+|\mu_l|)/\eps^2$ for $l=-\frac{N}{2},\ldots,\frac{N}{2}-1$, we obtain
\begin{align}\label{lm:local eq3}
\|\xi^n\|_{H^2}^2\lesssim&\tau^6\bigg[\|\partial_{ss}f_+^n\|_{L^\infty([0,\tau];H^2)}^2+
\|\partial_{ss}f_-^n\|_{L^\infty([0,\tau];H^2)}^2+
\|\partial_{ss}g_+^n\|_{L^\infty([0,\tau];H^2)}^2\nonumber\\
&+\|\partial_{ss}g_-^n\|_{L^\infty([0,\tau];H^2)}^2+\frac{1}{\eps^8}\|w^n\|_{L^\infty([0,\tau];H^4)}^2
+\frac{1}{\eps^4}\|\partial_sw^n\|_{L^\infty([0,\tau];H^3)}^2\nonumber\\
&+\|\partial_{ss}w^n\|_{L^\infty([0,\tau];H^2)}^2\bigg]\lesssim\frac{\tau^6}{\eps^4},\qquad 0<\tau\leq\tau_1.
\end{align}
Similarly, we can get
\begin{equation} \label{lm:local eq4}
\|\partial_x\xi^n\|_{H^2}^2\lesssim\frac{\tau^6}{\eps^6}, \qquad
\|\dot{\xi}^n\|_{H^2}^2\lesssim\frac{\tau^6}{\eps^8},\qquad 0<\tau\leq\tau_1.
\end{equation}
Again, substituting (\ref{lm:local eq3}) and (\ref{lm:local eq4}) into (\ref{error engery}) with
$e_N^n=\xi^{n}$  and $\dot{e}_N^n=\dot{\xi}^{n}$, we immediately
get the second inequality in (\ref{lm:local_error_L result}).
\end{proof}

For any $\bv\in Y_N$, we denote $v_{-1}=v_{N-1}$ and $v_{n+1}=v_1$ and then
define the difference operators $\delta_x^+\bv\in Y_N$ and $\delta_x^2\bv\in Y_N$ as
\[\delta_x^+\bv_j=\frac{v_{j+1}-v_j}{h}, \qquad \delta_x^2\bv_j=\frac{v_{j+1}-2v_j+v_{j-1}}{h^2},
\qquad j=0,1,\ldots,N.\]
In addition, we define the following norms as
$\|\bv\|_{Y,1}^2=\|\bv\|_{l^2}^2+\|\delta_x^+\bv\|_{l^2}^2$ and
$\|\bv\|_{Y,2}^2=\|\bv\|_{l^2}^2+\|\delta_x^+\bv\|_{l^2}^2+\|\delta_x^2\bv\|_{l^2}^2$
and it is easy to see that
\begin{eqnarray}\label{normequ986}
\qquad \quad \|I_N\bv\|_{H^1}\lesssim \|\bv\|_{Y,1}\lesssim \|I_N\bv\|_{H^1},
\quad \|I_N\bv\|_{H^2}\lesssim \|\bv\|_{Y,2}\lesssim \|I_N\bv\|_{H^2},
\quad \forall \bv\in Y_N.
\end{eqnarray}
Let $\bz_\pm^0\in Y_N$, $\dot{\bz}_\pm^0\in Y_N$, $\bff_\pm^0\in Y_N$, $\dot{\bff}_\pm^0\in Y_N$,
$\bgg_\pm^0\in Y_N$ and $\dot{\bgg}_\pm^0\in Y_N$ with
$z_{\pm,j}^0$, $\dot{z}_{\pm,j}^0$,  $f_{\pm,j}^0$, $\dot{f}_{\pm,j}^0$,
$g_{\pm,j}^0$ and $\dot{g}_{\pm,j}^0$, respectively,
for $j=0,1,\ldots,N$ be defined in (\ref{MTI-FP E}),
and  define the following error functions $\bee_{z_{\pm}}^n\in Y_N$,
$\dot{\bee}_{z_{\pm}}^n\in Y_N$, $\bee_{f_{\pm}}^n\in Y_N$,
$\dot{\bee}_{f_{\pm}}^n\in Y_N$, $\bee_{g_{\pm}}^n\in Y_N$ and
$\dot{\bee}_{g_{\pm}}^n\in Y_N$ as
\begin{equation}\label{err7896}
\left\{
\begin{split}
&e_{z_{\pm},j}^n=z_\pm^n(x_j,0)-z_{\pm,j}^0, \quad \dot{e}_{z_{\pm},j}^n=\partial_s
z_\pm^n(x_j,0)-\dot{z}_{\pm,j}^0,\\
&e_{f_{\pm},j}^n=f_\pm^n(x_j,0)-f_{\pm,j}^0, \quad
\dot{e}_{f_{\pm},j}^n=\partial_s
f_\pm^n(x_j,0)-\dot{f}_{\pm,j}^0,\qquad 0\le j\le N,\\
&e_{g_{\pm},j}^n=g_\pm^n(x_j,0)-g_{\pm,j}^0, \quad \dot{e}_{g_{\pm},j}^n=\partial_s
g_\pm^n(x_j,0)-\dot{g}_{\pm,j}^0.
\end{split}
\right.
\end{equation}

\begin{lemma}[Interpolation error]\label{lm: nonlinear0}
Under the assumption (A) and assume (\ref{MTI sol bound}) holds (which will be proved by
induction later), then we have
\begin{equation}\label{intpl567}
\begin{split}
&\|I_N\bee_{f_{\pm}}^n\|_{H^2}
+\|I_N\bee_{g_{\pm}}^n\|_{H^2}\lesssim \|e_N^n\|_{H^2}+\eps^2\|\dot{e}_N^n\|_{H^2}+h^{m_0},\\
&\|I_N\dot{\bee}_{f_{\pm}}^n\|_{H^2}+\|I_N\dot{\bee}_{g_{\pm}}^n\|_{H^2}\lesssim \frac{1}{\tau}\left(\|e_N^n\|_{H^2}+\eps^2\|\dot{e}_N^n\|_{H^2}+h^{m_0}+\tau^2\right).
\end{split}
\end{equation}
\end{lemma}

\begin{proof} From (\ref{err7896}), (\ref{normequ986}), (\ref{MTI-FP E}) and (\ref{fg_pm w}), we have
\begin{eqnarray}
\|I_N\bee_{f_{\pm}}^n\|_{H^2}&\lesssim&\|\bee_{f_{\pm}}^n\|_{Y,2}\nonumber\\
&\leq&\int_0^1\left[\left\|\partial_{z_+}f_\pm\left( \bz_+^\theta,\bz_-^n\right)\cdot\bee_{z_{+}}^n
\right\|_{Y,2}+\left\|\partial_{z_-}f_\pm\left(\bz_+^0,\bz_-^\theta\right)
\cdot\bee_{z_{-}}^n\right\|_{Y,2}\right]d\theta,
\label{lm: nonlinear0 eq1}
\end{eqnarray}
where $\bz_\pm^\theta\in Y_N$ and $\bz_\pm^n\in Y_N$ are defined as $\bz_{\pm,j}^\theta=\theta z_\pm^n(x_j,0)+(1-\theta)z_{\pm,j}^0$ and $\bz_\pm^n=z_{\pm,j}^n$, respectively,
for $j=0,1,\ldots,N$ and $0\le \theta\le 1$.
Under the assumption (\ref{MTI sol bound}) and using the  Sobolev's inequality, we get
\begin{eqnarray*}
&&\int_0^1\left\|\partial_{z_+}f_\pm\left( \bz_+^\theta,\bz_-^n\right)\cdot\bee_{z_{+}}^n
\right\|_{Y,2}d\theta\lesssim\left\|\bee_{z_{+}}^n\right\|_{l^\infty}\cdot
\int_0^1\left\|\delta_x^2\partial_{z_+}f_\pm\left( \bz_+^\theta,\bz_-^n\right)\right\|_{l^2}d\theta\nonumber\\
&&\ \ +\left\|\bee_{z_{+}}^n\right\|_{Y,1}\cdot\int_0^1\left\|\delta_x^+\partial_{z_+}f_\pm\left( \bz_+^\theta,\bz_-^n\right)\right\|_{l^\infty}d\theta
+\left\|\bee_{z_{+}}^n\right\|_{Y,2}\cdot\int_0^1\left\|\partial_{z_+}f_\pm\left( \bz_+^\theta,\bz_-^n\right)\right\|_{l^\infty}d\theta\nonumber\\
&&\lesssim\left\|\bee_{z_{+}}^n\right\|_{Y,2}.
\end{eqnarray*}
Similarly, we have
\[\int_0^1\left\|\partial_{z_-}f_\pm\left(z_+^0,\bz_-^\theta\right)\cdot\bee_{z_{-}}^n
\right\|_{Y,2}d\theta\lesssim\left\|\bee_{z_{-}}^n\right\|_{Y,2}.
\]
Plugging the above two inequalities into (\ref{lm: nonlinear0 eq1}), we get
\begin{eqnarray}\label{efpmg123}
\|I_N\bee_{f_{\pm}}^n\|_{H^2}&\lesssim&\left\|\bee_{z_{+}}^n\right\|_{Y,2}+\left\|\bee_{z_{-}}^n\right\|_{Y,2}
\lesssim\left\|I_N\bee_{z_{+}}^n\right\|_{H^2}+\left\|I_N\bee_{z_{-}}^n\right\|_{H^2}\nonumber\\
&\lesssim&\left\|I_Nu(\cdot,t_n)-u_I^n\right\|_{H^2}+\eps^2\left\|I_N\partial_tu(\cdot,t_n)-
\dot{u}_I^n\right\|_{H^2}\nonumber\\
&\lesssim&\left\|e_N^n\right\|_{H^2}+\eps^2\left\|\dot{e}_N^n\right\|_{H^2}+h^{m_0}.
\end{eqnarray}
In addition, combining (\ref{FSW-i21}) and (\ref{MTI-FP E}), we obtain
\begin{align}\label{efd8754}
 \|I_N\dot{\bee}_{f_{\pm}}^n\|_{H^2}\lesssim&\|\dot{\bee}_{f_{\pm}}^n\|_{Y,2}
\lesssim\left\|I_N\bee_{z_{+}}^n\right\|_{H^2}+\left\|I_N\bee_{z_{-}}^n\right\|_{H^2}
+\left\|I_N\dot{\bee}_{z_{+}}^n\right\|_{H^2}+\left\|I_N\dot{\bee}_{z_{-}}^n\right\|_{H^2}\nonumber\\
\lesssim&\left\|e_N^n\right\|_{H^2}+\eps^2\left\|\dot{e}_N^n\right\|_{H^2}+h^{m_0}+\left\|\partial_sz_{+}^n(\cdot,0)-I_N\dot{z}_+^0\right\|_{H^2}\nonumber\\
&+\left\|\partial_sz_{-}^n(\cdot,0)-I_N\dot{z}_-^0\right\|_{H^2}.
\end{align}
Noticing $\partial_sz_{\pm}^n(x,0)=\frac{i}{2}[-\partial_{xx}z_\pm^n(x,0)+f_\pm^n(z_+(x,0),z_-(x,0))]$,
 we have in Fourier space
\begin{align}\label{eq:iv-err}
\ \widehat{(\partial_sz_{\pm}^n)_l}
=&\frac{i}{2}\left[\mu_l^2 \widehat{(z_\pm^n)_l}+\widehat{(f_{\pm}^n)_l}\right]\\
=&\frac{i}{2}\left[2\frac{\sin(\frac12\tau\mu_l^2)}{\tau}\widehat{(z_\pm^n)_l}+
\widehat{(f_{\pm}^n)_l}\right]+\frac{i\mu_l^2}{2}\left(1-\frac{\sin(\frac12\tau\mu_l^2)}{\frac{1}{2}\tau\mu_l^2}\right)\widehat{(z_\pm^n)_l}.
\nonumber
\end{align}
Since the 'sinc' function $\text{sinc}(s)=\frac{\sin s}{s}$ if $s\neq0$ and $\text{sinc}(0)=1$ has the property that
${\text{sinc}}^\prime(0)=0$ and all the derivatives of {\text{sinc}} are bounded, we find
\[\left|1-\frac{\sin\left(\frac{1}{2}\tau\mu_l^2\right)}{\frac{1}{2}\tau\mu_l^2}\right|=\left|\text{sinc}(0)-\text{sinc}
\left(\frac12\tau\mu_l^2\right)\right|
\leq \frac12\tau\mu_l^2 \left\|\text{sinc}^{\prime}(\cdot)\right\|_{L^\infty}.\]
Then from (\ref{MTI-FP E}) and Lemma \ref{lm: prior} we have for small $\tau$,
\begin{equation}\label{lm: nonlinear eq6}
\left\|\partial_sz_{\pm}^n(\cdot,0)-I_N\dot{z}_\pm^0\right\|_{H^2}\lesssim \frac{1}{\tau}
\left(\left\|e_N^n\right\|_{H^2}+\eps^2\left\|\dot{e}_N^n\right\|_{H^2}+h^{m_0}\right)+\tau\|z_{\pm}^n(\cdot,0)\|_{H^6}.
\end{equation}
Plugging (\ref{lm: nonlinear eq6}) into (\ref{efd8754}), we get
$$\|I_N\dot{\bee}_{f_{\pm}}^n\|_{H^2}\lesssim \frac{1}{\tau}\left(\|e_N^n\|_{H^2}+\eps^2\|\dot{e}_N^n\|_{H^2}+h^{m_0}+\tau^2\right).$$
Similarly, we can get the estimate results for $\|I_N\bee_{g_{\pm}}^n\|_{H^2}$ and $\|I_N\dot{\bee}_{g_{\pm}}^n\|_{H^2}$.
Combining all, we immediately get (\ref{intpl567}).
\end{proof}

Defining the errors from the nonlinear terms as
\begin{equation}\label{eta def}
\eta^n(x):=\sum_{l=-N/2}^{N/2-1}\widetilde{\eta_l^n}\;\fe^{i\mu_l(x-a)},\quad \dot{\eta}^n(x):=\sum_{l=-N/2}^{N/2-1}\widetilde{\dot{\eta}_l^n}\;\fe^{i\mu_l(x-a)},
\quad x\in\overline{\Omega},\quad n\ge0,
\end{equation}
where
\begin{equation}\label{eta l def}
\widetilde{\eta_l^n}=\widehat{\mathcal{G}_l^n}-\widetilde{G_l^n},\qquad
\widetilde{\dot{\eta}_l^n}=\widehat{\mathcal{\dot{G}}_l^n}-\widetilde{\dot{G}_l^n},
\qquad l=-\frac{N}{2},\ldots \frac{N}{2}-1,
\end{equation}
then we have

\begin{lemma}[Estimates on $\eta^n$ and $\dot{\eta}^n$]\label{lm: nonlinear}
Under the same assumptions as in Lemma \ref{lm: nonlinear0},
 we have for any $0<\tau\leq\tau_1$,
\begin{equation}\label{lm: nonlinear eq}
\mcE\left(\eta^{n},\dot{\eta}^{n}\right)\lesssim \tau^2\mcE\left(e_N^{n},\dot{e}_N^{n}\right)+\frac{\tau^2 h^{2m_0}}{\eps^2}+\frac{\tau^6}{\eps^2},\quad n=0,1,\ldots,\frac{T}{\tau}-1.
\end{equation}
\end{lemma}

\begin{proof} Denote
\begin{equation}\label{efg965}
\left\{\begin{split}
&e_{f_\pm}^n(x)=f_\pm^n(x)-(I_N\bff_\pm^0)(x), \quad \dot{e}_{f_\pm}^n(x)=\partial_s f_\pm^n(x)-(I_N\dot{\bff}_\pm^0)(x), \\
&e_{g_\pm}^n(x)=f_\pm^n(x)-(I_N\bff_\pm^0)(x),\quad
\dot{e}_{g_\pm}^n(x)=\partial_s g_\pm^n(x)-(I_N\dot{\bgg}_\pm^0)(x), \quad x\in\Omega.
\end{split}
\right.
\end{equation}
For $l=-N/2,\ldots,N/2-1$, from  (\ref{efg965}),
(\ref{eta l def}) and (\ref{subeq4}),
using the triangle inequality, we have
\begin{eqnarray}\label{lm: nonlinear eq5}
|\eta^n_l|&\leq&|c_l(\tau)|\left[\left|\widehat{(e_{f_+}^n)}_l\right|+
\left|\widehat{(\overline{e_{f_-}^n})}_l\right|\right]
+|d_l(\tau)|\left[\left|\widehat{(\dot{e}_{f_+}^n)}_l\right|+
\left|\widehat{(\overline{\dot{e}_{f_-}^n})}_l\right|\right]\nonumber\\
&&+|p_l(\tau)|\left[\left|\widehat{(e_{g_+}^n)}_l\right|
+\left|\widehat{(\overline{e_{g_-}^n})}_l\right|\right]
+|q_l(\tau)|\left[\left|\widehat{(\dot{e}_{g_+}^n)}_l\right|+
\left|\widehat{(\overline{\dot{e}_{g_-}^n})}_l\right|\right].
\end{eqnarray}
From (\ref{cldl765}) directly, we have
\begin{equation}\label{add0}
\begin{split}
&|c_l(\tau)|+|p_l(\tau)|\lesssim\frac{\tau}{\sqrt{1+\mu_l^2\eps^2}}\lesssim \tau,\qquad
\mu_l(|c_l(\tau)|+|p_l(\tau)|)\lesssim\frac{\tau}{\eps},\\
&|d_l(\tau)|+|q_l(\tau)|\lesssim\frac{\tau^2}{\sqrt{1+\mu_l^2\eps^2}}\lesssim \tau^2,\qquad
\mu_l(|d_l(\tau)|+|q_l(\tau)|)\lesssim\frac{\tau^2}{\eps}.
\end{split}
\end{equation}
Inserting (\ref{add0}) into (\ref{lm: nonlinear eq5}) and using the Cauchy's inequality,
we obtain
\begin{eqnarray}\label{lm: nonlinear eq1}
\|\eta^n\|_{H^2}^2&\lesssim&\tau^2\bigg[\left\|P_Ne_{f_+}^n\right\|_{H^2}^2+
\left\|P_Ne_{f_-}^n\right\|_{H^2}^2
+\left\|P_Ne_{g_+}^n\right\|_{H^2}^2+\left\|P_Ne_{g_-}^n\right\|_{H^2}^2\bigg]\nonumber\\
&&+\tau^4\bigg[\left\|P_N\dot{e}_{f_+}^n\right\|_{H^2}^2+\left\|P_N\dot{e}_{f_-}^n\right\|_{H^2}^2
+\left\|P_N\dot{e}_{g_+}^n\right\|_{H^2}^2+\left\|P_N\dot{e}_{g_-}^n\right\|_{H^2}^2\bigg]\nonumber\\
&\lesssim&\tau^2\bigg[\left\|I_Ne_{f_+}^n\right\|_{H^2}^2+
\left\|I_Ne_{f_-}^n\right\|_{H^2}^2
+\left\|I_Ne_{g_+}^n\right\|_{H^2}^2+\left\|I_Ne_{g_-}^n\right\|_{H^2}^2\bigg]\nonumber\\
&&+\tau^4\bigg[\left\|I_N\dot{e}_{f_+}^n\right\|_{H^2}^2+\left\|I_N\dot{e}_{f_-}^n\right\|_{H^2}^2
+\left\|I_N\dot{e}_{g_+}^n\right\|_{H^2}^2+\left\|I_N\dot{e}_{g_-}^n\right\|_{H^2}^2\bigg]+\tau^2h^{2m_0}\nonumber\\
&\lesssim&\tau^2\|e_N^n\|_{H^2}^2+\tau^2\eps^4\|\dot{e}_N^n\|_{H^2}^2+\tau^2h^{2m_0}+\tau^6,
\end{eqnarray}
and
\begin{equation}\label{lm: nonlinear eq3}
\|\partial_x\eta^n\|_{H^2}^2\lesssim\frac{\tau^2}{\eps^2}\|e_N^n\|_{H^2}^2+\tau^2\eps^2\|\dot{e}_N^n\|_{H^2}^2
+\frac{\tau^2h^{2m_0}}{\eps^2}+\frac{\tau^6}{\eps^2}.
\end{equation}
Similarly,
\begin{equation}\label{lm: nonlinear eq2}
\eps^2\|\dot{\eta}^n\|_{H^2}^2
\lesssim\frac{\tau^2}{\eps^2}\|e_N^n\|_{H^2}^2+\tau^2\eps^2\|\dot{e}_N^n\|_{H^2}^2
+\frac{\tau^2h^{2m_0}}{\eps^2}+\frac{\tau^6}{\eps^2}.
\end{equation}
Combining (\ref{lm: nonlinear eq1}), (\ref{lm: nonlinear eq2}) and (\ref{error engery})
 we immediately obtain (\ref{lm: nonlinear eq}).
\end{proof}

\medskip

{\it Proof of Theorem \ref{main thm}.} The proof will be proceeded by the method of mathematical induction
and the energy method. For $n=0$, from the initial data in the MTI-FP (\ref{MTI-FP S})-(\ref{MTI-FP E})
method and noticing the assumption (A), we have
\[\|e^0\|_{H^2}+\eps^2\|\dot{e}^0\|_{H^2}=\|\phi_1-I_N\phi_1\|_{H^2}+\|\phi_2-I_N\phi_2\|_{H^2}\lesssim h^{m_0+2}\lesssim h^{m_0}.\]
In addition, using the triangle inequality, we know that there exists $h_1>0$ independent of $\eps$ such that
for $0<h\leq h_1$ and $\tau>0$
\[\|u_I^0\|_{H^2}\leq\|\phi_1\|_{H^2}+\|e^0\|_{H^2}\leq C_0+1,\quad\  \|\dot{u}_I^0\|_{H^2}\leq\frac{\|\phi_2\|_{H^2}}{\eps^2}+\|\dot{e}^0\|_{H^2}\leq\frac{C_0+1}{\eps^2}.\]
Thus (\ref{MTI error bound1})-(\ref{MTI sol bound}) are valid for $n=0$.
Now  we assume that (\ref{MTI error bound1})-(\ref{MTI sol bound}) are valid for $0\leq n\leq m-1\leq T/\tau-1$.
Substracting (\ref{subeq1}) from (\ref{observ of scheme}), we have
\begin{subequations}\label{lm: error prog eq1}
\begin{eqnarray}
&&\qquad\quad \widehat{(e^{n+1})}_l=\widehat{u}_l(t_{n+1})-\widetilde{(u_I^{n+1})}_l=
\cos(\omega_l\tau)\widehat{(e^{n})}_l+
\frac{\sin(\omega_l\tau)}{\omega_l}\widehat{(\dot{e}^{n})}_l+\widehat{\xi_l^n}-\widetilde{\eta^n_l},\\
&&\qquad\quad \widehat{(\dot{e}^{n+1})}_l=\widehat{u}_l'(t_{n+1})-\widetilde{(\dot{u}_I^{n+1})}_l=
-\omega_l\sin(\omega_l\tau)\widehat{(e^{n})}_l+\cos(\omega_l\tau)\widehat{(\dot{e}^{n})}_l
+\widehat{\dot{\xi}_l^n}-\widetilde{\dot{\eta}^n_l}.
\end{eqnarray}
\end{subequations}
Using the Cauchy's inequality, we obtain
\begin{subequations}
\begin{align}
&\left|\widehat{(e^{n+1})}_l\right|^2\leq(1+\tau)\left|\cos(\omega_l\tau)\widehat{(e^{n})}_l
+\frac{\sin(\omega_l\tau)}{\omega_l}\widehat{(\dot{e}^{n})}_l\right|^2+\frac{1+\tau}
{\tau}\left|\widehat{\xi_l^n}-\widetilde{\eta^n_l}\right|^2,
\label{lm:error prog eq2}\\
&\left|\widehat{(\dot{e}^{n+1})}_l\right|^2\leq(1+\tau)\left|\cos(\omega_l\tau)\widehat{(\dot{e}^{n})}_l
-\omega_l\sin(\omega_l\tau)\widehat{(e^{n})}_l\right|^2
+\frac{1+\tau}{\tau}\left|\widehat{\dot{\xi}_l^n}-\widetilde{\dot{\eta}^n_l}\right|^2.\label{lm:error prog eq3}
\end{align}
\end{subequations}
Multiplying (\ref{lm:error prog eq2}) and (\ref{lm:error prog eq3}) by
 $(\mu_l^2+\frac{1}{\eps^2})(1+\mu_l^2+\mu_l^4)$ and
$\eps^2(1+\mu_l^2+\mu_l^4)$, respectively, and then  summing them up for $l=-N/2,\ldots, N/2-1$,
we obtain
\[\mathcal{E}(e^{n+1}_N,\dot{e}^{n+1}_N)\leq (1+\tau)\mathcal{E}(e^{n}_N,\dot{e}^{n}_N)+\frac{1+\tau}{\tau}
\mathcal{E}(\xi^n-\eta^n,\dot{\xi}^n-\dot{\eta}^n).\]
Using the Cauchy's inequality, we get
\begin{equation}\mathcal{E}(e^{n+1}_N,\dot{e}^{n+1}_N)-\mathcal{E}(e^{n}_N,
\dot{e}^{n}_N)\lesssim\tau\mathcal{E}(e^{n}_N,\dot{e}^{n}_N)
+\frac{1+\tau}{\tau}\left[\mathcal{E}(\xi^n,\dot{\xi}^n)
+\mathcal{E}(\eta^n,\dot{\eta}^n)\right].\label{proof: eq1}
\end{equation}
Inserting  (\ref{lm: nonlinear eq}) and the second inequality in (\ref{lm:local_error_L result})
 into (\ref{proof: eq1}), we get
\[\mcE\left(e^{n+1}_N,\dot{e}^{n+1}_N\right)-\mcE\left(e^{n}_N,\dot{e}^{n}_N\right)
\lesssim\tau\mcE\left(e^{n}_N,\dot{e}^{n}_N\right)+\frac{\tau h^{2m_0}}{\eps^2}+\frac{\tau^5}{\eps^6}.\]
Summing the above inequality for $0\le n\le m-1$ and then applying the discrete Gronwall's inequality, we have
\begin{equation}\label{emn865}
\mcE\left(e^{m}_N,\dot{e}^{m}_N\right)\lesssim\frac{h^{2m_0}}{\eps^2}+\frac{\tau^4}{\eps^6}.
\end{equation}
Similarly, by using the first inequality in  (\ref{lm:local_error_L result}), we obtain
\begin{equation}\label{emn835}
\mcE\left(e^{m}_N,\dot{e}^{m}_N\right)\lesssim \frac{h^{2m_0}}{\eps^2}+\frac{\tau^4}{\eps^2}+\eps^2.
\end{equation}
Combining (\ref{error engery}), (\ref{proof: eq0}), (\ref{emn865}) and (\ref{emn835}),
we get that (\ref{MTI error bound1}) is valid for $n=m$, which implies \cite{Deg,Jin}
\[\|e^m\|_{H^2}+\eps^2\|\dot{e}^m\|_{H^2}\leq h^{m_0}+\tau.\]
Using the triangle inequality, we obtain that these exist $h_2>0$ and $\tau_2>0$ independent of
$\eps$ such that
\[\begin{split}
&\|u_I^m\|_{H^2}\leq \|u(\cdot,t_m)\|_{H^2}+\|\dot{e}^m\|_{H^2}\leq C_0+1,\\
&\|\dot{u}_I^m\|_{H^2}\leq \|\partial_tu(\cdot,t_m)\|_{H^2}+\|\dot{e}^m\|_{H^2}\leq \frac{C_0+1}{\eps^2},
\qquad 0<h\le h_2,\quad 0<\tau\le \tau_2.
\end{split}
\]
Thus (\ref{MTI sol bound}) is also valid for $n=m$. Then the proof is completed
by chosen $\tau_0=\min\{\tau_1,\tau_2\}$ and $h_0=\min\{h_1,h_2\}$.
$\Box$

\begin{remark}
Here we emphasize that Theorem \ref{main thm} holds in 2D and 3D and the above approach can be directly extended to the higher dimensions
without any extra efforts.
The only thing needs to be taken care of is the Sobolev inequality used in Lemma \ref{lm: nonlinear0} in 2D and 3D,
\begin{equation}
\begin{split}
&\|u\|_{L^\infty(\Omega)}\leq C\|u\|_{H^2(\Omega)},\quad \text{in 2D and 3D},\\
&\|u\|_{W^{1,p}(\Omega)}\leq C\|u\|_{H^2(\Omega)},\quad 1<p<6 \text{ in 2D and 3D},
\end{split}
\end{equation}
where $\Omega$ is a bounded domain in 2D or 3D.  By using assumption
\eqref{MTI sol bound}, Lemma \ref{lm: nonlinear0} will still hold in 2D and 3D.
\eqref{MTI sol bound} and error bounds can be proved by induction since our scheme
is explicit.
\end{remark}

Under a weaker assumption of the regularity
\[({\rm B}) \quad u\in C^1\left([0,T];H_p^{m_0+3}(\Omega)\right),
\quad \left\|u\right\|_{L^\infty([0,T]; H^{m_0+3})}+\eps^2
\left\|\partial_t u\right\|_{L^\infty([0,T]; H^{m_0+3})}\lesssim 1,
\]
with $m_0\geq2,$ we can have the $H^1$-error estimates of the MTI-FP method by a very similar proof with all the $H^2$-norms in above changed into $H^1$-norms.
\begin{theorem}\label{main thm 2}
Under the assumption (B), there exist two constants $0<h_0\leq1$ and $0<\tau_0\leq1$
sufficiently small and independent of $\eps$ such that, for any $0<\eps\leq1$, when
 $0<h\leq h_0$ and $0<\tau\leq\tau_0$, we have
\begin{align}
&\left\|e^n\right\|_{H^1}+\eps^2\left\|\dot{e}^n\right\|_{H^1}\lesssim h^{m_0}+\frac{\tau^2}{\eps^2},
\qquad \left\|e^n\right\|_{H^1}+\eps^2\left\|\dot{e}^n\right\|_{H^1}\lesssim h^{m_0}+\tau^2+\eps^2, \\
&\left\|u^n\right\|_{l^\infty}\leq C_0+1,\quad \left\|\dot{u}^n\right\|_{l^\infty}\leq \frac{C_0+1}{\eps^2},\qquad 0\leq n\leq\frac{T}{\tau}.
\end{align}
\end{theorem}
\begin{remark}
In 1D case, Theorem \ref{main thm 2} holds without any CFL-type conditions. However for higher dimensional cases, i.e. $d=2$ or $d=3$, due to the use of
inverse inequality to provide the $l^\infty$ control of the numerical solution \cite{Bao1}, 
one has to impose the technical condition
\begin{equation*}
\tau\lesssim \rho_d(h),\qquad \hbox{with} \quad \rho_d(h)=\left\{\begin{array}{ll}   1/|\ln h|,\quad &d=2,\\ \sqrt{h},\quad &d=3.\\
\end{array}\right.
\end{equation*}
If the solution of the KG is smooth enough, we can always turn to Theorem \ref{main thm} and such CFL type conditions are
 unnecessary.
\end{remark}
\begin{remark}
If the periodic boundary condition for the KG equation (\ref{KG trun}) is replaced by the
homogeneous Dirichlet or Neumann boundary condition, then the MTI-FP method and its error estimates
are still valid provided that the Fourier basis is replaced by sine or cosine basis.
\end{remark}

\begin{remark}
If the cubic nonlinearity in the KG equation (\ref{KG}) is replaced by a general gauge invariant nonlinearity,
the general MTI-FP method can be designed similar to those in  \cite{Zhao}.
\end{remark}

\section{Numerical results}\label{sec: num results}
In this section, we present numerical results of the MTI-FP method to confirm
our error estimates. In order to do so, we take $d=1$ and $f(u)=|u|^2u$ in
(\ref{KG}) and choose the initial data as
\[\phi_1(x)=(1+i)\fe^{-x^2/2},\qquad \phi_2(x)=\frac{3\fe^{-x^2/2}}{2\eps^2},\qquad x\in{\mathbb R}.\]
The problem is solved on a bounded interval $\Omega=[-16,16]$, i.e. $b=-a=16$,
which is large enough to guarantee that the periodic boundary condition does not
 introduce a significant aliasing error relative to the original problem.
To quantify the error, we introduce two error functions:
\begin{equation*}
\fe^{\tau,h}_\eps(T):=\left\|u(\cdot,T=M\tau)-u^M_I\right\|_{H^2},\qquad
\fe^{\tau,h}_{\infty}(T):=\max_{\eps}\left\{\fe^{\tau,h}_\eps(T)\right\}.
\end{equation*}
Since the analytical solution to this problem is not available,
so the `exact' solution is obtained numerically by the MTI-FP method
(\ref{MTI-FP S})-(\ref{MTI-FP E}) with very fine mesh $h=1/32$ and
time step $\tau=5\times 10^{-6}$.
Tab. \ref{tb: MTIFP spa} shows the spatial error of MTI-FP method at $T=1$
under different $\eps$ and $h$ with a very small time step $\tau=5\times 10^{-6}$
such that the discretization error in time is negligible.
Tab. \ref{tb: MTIFP tem} shows the temporal error of MTI-FP method at $T=1$
under different $\eps$ and $\tau$ with a small mesh size $h=1/8$ such that
the discretization error in space is negligible.

\begin{table}[t!]
\tabcolsep 0pt
\caption{Spatial error analysis: $\fe^{\tau,h}_\eps(T=1)$ with  $\tau=5\times 10^{-6}$
for different $\eps$ and $h$.}\label{tb: MTIFP spa}
\begin{center}
\begin{tabular*}{1\textwidth}{@{\extracolsep{\fill}}lllll}
\hline
$\fe^{\tau,h}_\eps(T)$         & $h_0=1$	      &$h_0/2$	        &$h_0/4$	        &$h_0/8$	\\
\hline
$\eps_0=0.5$	               &1.65E\,--\,1	  &3.60E\,--\,3   	&1.03E\,--\,6	    &7.34E\,--\,11\\
$\eps_0/2^1$	               &2.65E\,--\,1	  &9.70E\,--\,3	    &9.07E\,--\,7	    &5.03E\,--\,11\\
$\eps_0/2^2$	               &9.02E\,--\,1	  &1.34E\,--\,2	    &1.73E\,--\,7	    &4.60E\,--\,11\\
$\eps_0/2^3$	               &1.13E+0	          &2.98E\,--\,2	    &2.25E\,--\,7	    &4.10E\,--\,11\\
$\eps_0/2^4$	               &4.67E\,--\,1	  &3.14E\,--\,2	    &1.79E\,--\,7	    &4.78E\,--\,11\\
$\eps_0/2^5$	               &7.41E\,--\,1	  &2.73E\,--\,2 	&2.50E\,--\,7	    &5.49E\,--\,11\\
$\eps_0/2^7$	               &7.41E\,--\,1	  &2.62E\,--\,2	    &2.12E\,--\,7	    &4.96E\,--\,11\\
$\eps_0/2^{9}$	               &6.33E\,--\,1	  &3.57E\,--\,2	    &1.92E\,--\,7	    &5.04E\,--\,11\\
$\eps_0/2^{11}$	               &9.19E\,--\,1	  &2.44E\,--\,2	    &2.19E\,--\,7	    &6.18E\,--\,11\\
$\eps_0/2^{13}$	               &1.18E+0	          &2.38E\,--\,2	    &2.59E\,--\,7	    &5.86E\,--\,11\\
\hline
\end{tabular*}
\end{center}
\end{table}

\begin{table}[t!]
\tabcolsep 0pt
\caption{Temporal error analysis: $\fe^{\tau,h}_\eps(T=1)$ a
nd $\fe^{\tau,h}_{\infty}(T=1)$ with $h=1/8$ for different $\eps$ and $\tau$.}\label{tb: MTIFP tem}
\begin{center}
\begin{tabular*}{1\textwidth}{@{\extracolsep{\fill}}llllllll}
\hline
$\fe^{\tau,h}_\eps(T)$      & $\tau_0=0.2$	&$\tau_0/2^2$	
&$\tau_0/2^4$	&$\tau_0/2^6$	 &$\tau_0/2^8$  & $\tau_0/2^{10}$  &$\tau_0/2^{12}$\\
\hline
$\eps_0=0.5$	            &7.17E-1	    &5.72E-2	    &3.50E-3	
&2.14E-4	     &1.33E-5	    &8.14E-7	       &3.67E-8\\
rate                        &---	            &1.82	        &2.02	
&2.01	         &2.00	        &2.01	           &2.20\\ \hline
$\eps_0/2^1$	            &5.40E-1	    &1.58E-1	    &1.12E-2	
&6.74E-4	     &4.15E-5	    &2.54E-6	       &1.18E-7\\
rate                        &---	            &0.89	        &1.91	
&2.02	         &2.01	        &2.01	           &2.21\\ \hline
$\eps_0/2^2$	            &5.23E-1	    &1.47E-1	    &3.70E-2	
&2.70E-3	     &1.62E-4	    &9.87E-6	       &4.62E-7\\
rate                        &---	            &0.91	        &0.99	
&1.90	         &2.02	        &2.01	           &2.20\\ \hline
$\eps_0/2^3$	            &6.30E-1	    &6.28E-2	    &4.13E-2	
 &8.90E-3	     &6.51E-4	    &3.92E-5	       &1.82E-6\\
rate                        &---	            &1.66	        &0.30	
 &1.11	         &1.89	        &2.02	           &2.21\\ \hline
$\eps_0/2^4$	            &6.11E-1	    &3.00E-2	    &1.16E-2	
&1.05E-2	     &2.20E-3	    &1.60E-4	       &7.41E-6\\
rate                        &---	            &2.17	        &0.68	
   &0.07	         &1.13	        &1.89	           &2.21\\ \hline
$\eps_0/2^5$	            &6.17E-1	    &3.01E-2	    &2.70E-3	
 &2.90E-3	     &2.80E-3	    &5.26E-4	       &2.98E-5\\
rate                        &---	            &2.17	        &1.75	
 &-0.04	         &0.02	        &1.17	           &2.07\\ \hline
$\eps_0/2^7$	            &6.16E-1	    &2.90E-2	    &1.80E-3	
 &2.37E-4	     &1.37E-4	    &1.96E-4	       &1.91E-4\\
rate                        &---	            &2.20	        &2.01	
     &1.46	         &0.40	        &-0.26	           &0.02\\ \hline
$\eps_0/2^{9}$	            &6.13E-1	    &2.90E-2	    &1.69E-3	
  &1.12E-4	     &1.09E-5	    &5.51E-6	       &1.69E-6\\
rate                        &---	            &2.20	        &2.03	
 &1.96	         &1.68	        &0.49	           &0.85\\ \hline
$\eps_0/2^{11}$	            &6.16E-1	    &2.90E-2	    &1.69E-3	
  &1.05E-4	     &6.95E-6	    &9.97E-7	       &3.38E-7\\
rate                        &---	            &2.20	        &2.03	
    &2.00	         &1.96	        &1.40	           &0.78\\ \hline
$\eps_0/2^{13}$	            &6.20E-1	    &2.92E-2	    &1.69E-3		
 &1.06E-4	     &6.61E-6	    &3.94E-7	       &2.38E-8\\
rate                        &---	            &2.20	        &2.04
  &2.00	         &2.00	        &2.03	           &2.02\\ \hline \hline
$\fe^{\tau,h}_{\infty}(T)$  &7.17E-1	    &1.58E-1	    &4.13E-2	
 &1.05E-2	     &2.80E-3	    &5.26E-4	       &1.91E-4\\
rate                        &---	            &1.09	        &0.97	
    &0.99	         &1.00	        &1.15	           &0.74\\
\hline
\end{tabular*}
\end{center}
\end{table}

From Tabs. \ref{tb: MTIFP spa}-\ref{tb: MTIFP tem} and extensive additional results not shown here for
brevity, we can draw the following observations:

(i) The MTI-FP method is spectrally accurate in space, which is uniformly for $0<\eps\le1$
(cf. Tab. \ref{tb: MTIFP spa}).

(ii) The MTI-FP method converges uniformly and linearly
in time for $\eps\in(0,\tau]$ (cf. last row in Tab. \ref{tb: MTIFP tem}).
In addition, for each fixed $\eps=\eps_0>0$, when $\tau$ is small enough,
it converges quadratically in time (cf. each row in the upper triangle of Tab. \ref{tb: MTIFP tem});
and for each fixed $\eps$ small enough, when $\tau$ satisfies $0<\eps<\tau$,
it also converges quadratically in time (cf. each row in the lower triangle of Tab. \ref{tb: MTIFP tem}).

(iii) The MTI-FP method is uniformly accurate for all $\eps\in(0,1]$ under the mesh strategy
(or $\eps$-scalability) $\tau=O(1)$ and $h=O(1)$.

\section{Conclusions}\label{sec: conlusion}
A MTI-FP method was proposed and analyzed for solving the KG
equation with a dimensionless parameter $0<\eps\leq1$ which is inversely
proportional to the speed of light.
The key ideas for designing the MTI-FP method are based on
(i) carrying out a multiscale decomposition by frequency
at each time step with proper choice of transmission conditions between time steps,
and (ii) adapting the Fourier spectral for spatial discretization and
the EWI for integrating second-order highly oscillating ODEs.
Rigorous error bounds for the MTI-FP method were established,
which imply that the MTI-FP method converges uniformly and optimally
in space with spectral convergence rate, and uniformly  in time with linear convergence rate
for $\eps\in(0,1]$ and optimally with quadratic convergence rate in the regimes when either $\eps= O(1)$ or $0<\eps\leq\tau$. Numerical results confirmed these error bounds and suggested that they are sharp.

\end{document}